\documentclass[reqno,a4paper,11pt]{article}
\usepackage{amsmath,amsthm,dsfont,amsfonts,amssymb,fancyhdr}
\usepackage{enumerate}
\usepackage[usenames,dvipsnames]{color}
\usepackage{bbm,soul,supertabular,longtable,verbatim,extarrows}
\usepackage{titlesec}
\usepackage{graphicx}

\usepackage{cite}
\usepackage{dutchcal}
\usepackage{tikz}
\usepackage{booktabs,multirow,makecell}
\usepackage{authblk}
\usetikzlibrary{trees}

\usepackage{lipsum}
\usepackage{indentfirst}
\usepackage[colorlinks=true, allcolors=blue]{hyperref}
\usepackage[top=2.4cm,bottom=2.2cm,left=2.6cm,right=2cm]{geometry}

\newtheorem{theorem}{Theorem}[section]
\newtheorem{proposition}[theorem]{Proposition}
\newtheorem{lemma}[theorem]{Lemma}
\newtheorem{corollary}[theorem]{Corollary}
\newtheorem{conj}[theorem]{Conjecture}
\newtheorem{claim}[theorem]{Claim}
\theoremstyle{definition}

\newtheorem{remark}[theorem]{Remark}

\soulregister\cite7
\soulregister\citep7
\soulregister\citet7
\soulregister\ref7
\soulregister\pageref7

\begin{document}
\title{\textbf{The Lemmens-Seidel conjecture and forbidden subgraphs}}
\author[a]{Meng-Yue Cao}
\author[b,c,]{Jack H. Koolen\footnote{Corresponding author.}}
\author[d]{Yen-Chi Roger Lin}
\author[e]{Wei-Hsuan Yu}
\affil[a]{\footnotesize{School of Mathematical Sciences, Beijing Normal University, 19 Xinjiekouwai Street, Beijing, 100875, PR China.}}
\affil[b]{\footnotesize{School of Mathematical Sciences, University of Science and Technology of China, 96 Jinzhai Road, Hefei, 230026, Anhui, PR China.}}
\affil[c]{\footnotesize{Wen-Tsun Wu Key Laboratory of CAS, 96 Jinzhai Road, Hefei, 230026, Anhui, PR China}}
\affil[d]{\footnotesize{National Taiwan Normal University, Taipei, 11677, Taiwan}}
\affil[e]{\footnotesize{National Central University, Taoyuan, 32001, Taiwan }}
\date{}
\maketitle
\newcommand\blfootnote[1]{%
\begingroup
\renewcommand\thefootnote{}\footnote{#1}%
\addtocounter{footnote}{-1}%
\endgroup}
\blfootnote{E-mail addresses: {\tt cmy1325@163.com} (M.-Y. Cao), {\tt koolen@ustc.edu.cn} (J.H. Koolen),  {\tt yclinpa@gmail.com} (Y.-C. R. Lin), {\tt u690604@gmail.com} (W.-H. Yu).}

Dedicated to the $100$th birthday anniversary of Professor J. J. Seidel.

\begin{abstract}
  In this paper we show that the conjecture of Lemmens and Seidel of 1973 for systems of equiangular lines with common angle $\arccos(1/5)$ is true.
  Our main tool is forbidden subgraphs for smallest Seidel eigenvalue $-5$.
\end{abstract}

\section{Introduction}

A system of lines through the origin in the $r$-dimensional Euclidean space $\mathbb{R}^r$ is called equiangular if the angle between any pair of lines is the same. The study of equiangular lines has a long history and is related to  many things. For instance, the maximum size of equiangular lines is related to energy minimizing configurations~\cite{cohn2007universally}, line packing problems~\cite{fickus2016equiangular}, and tight spherical designs~\cite{bannai2009survey}. Several constructions of equiangular lines come from strongly regular graphs \cite{cameron2004strongly} and combinatorial designs \cite{taylor1971}. De Caen used association schemes to construct $\frac 2 9 (r+1)^2$ equiangular lines in $\mathbb{R}^r$ when $r=3 \cdot 2^{2t-1}-1$ for any positive integer $t$~\cite{de2000large}.We are interested in determining the maximum cardinality $N(r)$ of a system of equiangular lines in $\mathbb{R}^r$. Gerzon \cite{lemmens1973} proved that $N(r) \leqslant \frac{r(r+1)}{2}$ for all $r$.  However, so far the Gerzon bound is only known to be achieved for $r = 2, 3, 7$, and $23$.
If we have equiangular lines attaining the Gerzon bound, then we immediately have tight spherical 5-designs \cite{bannai2009survey}. The classification of tight spherical 5-designs has been open for decades and the main known necessary condition for the existence of tight
spherical 5-designs is $r=2,3$, or $r=(2k+1)^2-2$, where $k \in \mathbb{N}$.
The history of the study of equiangular lines can be traced back to Haantjes~\cite{haantjes1948}, who determined $N(3)$ and $N(4)$ in $1948$. After more than 70 years of study, the numbers $N(r)$ are now only known for $r\leqslant43$ except for $r = 17, 18, 19, 20$, and $42$.
This follows from the works of Van Lint and Seidel \cite{vanlint1966}, Lemmens and Seidel \cite{lemmens1973}, Barg and Yu \cite{barg2014}, and Greaves et al.~\cite{greaves2020equiangular}
We summarize the results in the following table.
For more references on recent progress of equiangular lines, readers may check \cite{balla2018equiangular, greaves2018equiangular, greaves2016, greaves2019equiangular, jiang2017forbidden, jiang2019fixedangle,   lin2018saturated}.

\begin{table}[ht]
	\centering
    \caption{Maximum cardinalities of equiangular lines for low dimensions}
    \label{tb:smallnd}
    \begin{tabular}{c|ccccccccc}
   		$r$    & 2 & 3--4 & 5  & 6  & 7--14 & 15 & 16     & 17 \\ \hline
        $M(r)$ & 3 & 6    & 10 & 16 & 28    & 36 & 40 & 48--49 \\
        \hline\hline
        $r$    & 18     & 19     & 20     & 21  & 22  & 23--41 & 42       & 43 \\ \hline
        $M(r)$ & 56--60 & 72--74 & 90--94 & 126 & 176 & 276    & 276--288 & 344
    \end{tabular}
\end{table}

Let $N_{\alpha}(r)$ be the maximum number of a system of equiangular lines in $\mathbb{R}^r$ with common angle $\arccos\alpha$. Neumann $(1973)$ showed that if $N_{\alpha}(r)>2r$, then $\frac{1}{\alpha}$ is an odd integer at least $3$.
Lemmens and Seidel \cite{lemmens1973} determined $N_{\frac{1}{3}}(r)$ for all $r\geqslant2$.
In particular, they showed that $N_{\frac{1}{3}}(r)=2r-2$ if $r\geqslant15$. They also proposed the following conjecture for the case $\frac{1}{\alpha}=5$.

\begin{conj}\label{conj:LS}
The maximum cardinality of a system of equiangular lines with angle $\arccos\frac{1}{5}$ in $\mathbb{R}^r$ is $276$ for $23 \leqslant r \leqslant 185$,
and $\lfloor \frac{3r-3}{2}\rfloor$ for $r \geqslant 185$.
\end{conj}

Neumaier \cite{neumaier1989graph} showed Conjecture \ref{conj:LS} for sufficient large $r$. He also claimed (without proof) that his method would work for $r\geqslant N_0$ where $ 2486\leqslant N_0 \leqslant 45374$. In this paper, we completely solve Conjecture \ref{conj:LS}. Balla, Dr\"{a}xler, Keevash and Sudakov\cite{balla2018equiangular} and Bukh \cite{bukh2016} conjectured an asymptotic version of Conjecture \ref{conj:LS} for other angles as follows:

\begin{conj}\label{conj:BDKS}
The maximum cardinality of a system of equiangular lines with angle $\arccos\alpha$, where $\frac{1}{\alpha}=2m+1$ is an odd integer at least $3$, is equal to $\frac{(m+1)(r+1)}{m}+O(1)$, for $r\rightarrow\infty$.
\end{conj}

Jiang and Polyanskii \cite{jiang2017forbidden} gave partial results for  Conjecture~\ref{conj:BDKS}, and it was completely solved by Jiang, Tidor, Yao, Zhang, and Zhao in a recent paper \cite{jiang2019fixedangle}.

\section{Outline of the paper}
All graphs in this paper are simple and undirected. For undefined terminologies, we refer to \cite{godsil2013,brouwer2011spectra}.

First, we transform the problem of determining $N_{\alpha}(r)$ into a linear algebra problem. To do so, we introduce Seidel matrices.

A Seidel matrix $S$ of order $n$ is a symmetric $(0, \pm1)$-matrix with $0$ on the diagonal and $\pm1$ otherwise. Seidel matrices and systems of equiangular lines, are related as follows (see for example, \cite[Section 11.1]{godsil2013}):

\begin{proposition}\label{eq}
Let $n>r\geqslant2$ be integers. There exists a system of $n$ equiangular lines in $\mathbb{R}^r$ with common angle $\arccos \alpha$ if and only if there exists a Seidel matrix $S$ of order $n$ such that $S$ has smallest eigenvalue at least $-\frac{1}{\alpha}$ and \rm{rk}$(S+\frac{1}{\alpha}\mathbf{I})\leqslant r$.
\end{proposition}

In this paper, we focus on the minimum rank of $S+\frac{1}{\alpha}\mathbf{I}$ for a fixed number $n$ rather than the maximum cardinality of a system of equiangular lines in $\mathbb{R}^r$ with common angle $\arccos\alpha$ for  fixed dimension $r$.
Our main result is as follows.

\begin{theorem}\label{maintheorem}
Let $S$ be a Seidel matrix of order $n$ with the smallest eigenvalue $-5$.
If $n\geqslant277$, then ${\rm rk}(S+5\mathbf{I})\geqslant\lfloor\frac{2n}{3}\rfloor+1$.
\end{theorem}
This theorem implies that Conjecture \ref{conj:LS} is true.

Our main tools are minimal forbidden subgraphs. We will first show that Theorem~\ref{maintheorem} is true when the independence number $\alpha([S])$ of the switching class $[S]$ of a Seidel matrix $S$ (for definitions see next section) is at least $49$. This uses, in addition to minimal forbidden subgraphs, also a rank argument, which is done in Section \ref{sec:49}. Then, in Section \ref{sec:triangle-free}, we concentrate on the case when $[S]$ contains a triangle-free graph. If the clique number $\omega([S])$ of $[S]$ is at least $5$, then Conjecture \ref{conj:LS} was already shown by Lemmens-Seidel~\cite{lemmens1973} and Lin-Yu~\cite{lin2018equiangular}. So we only need to show Theorem \ref{maintheorem} for the cases when $\omega([S])$ is at most $4$. Under this condition, we show Theorem \ref{maintheorem} is true when $\alpha([S])\geqslant29$ in Section~\ref{sec:29}. Then, in Section \ref{sec:pillar}, we apply the pillar method to the $(4,1)$-pillars and the $(4,2)$-pillars. Our bounds for the $(4,1)$-pillar and on the $(4,2)$-pillar are not yet sharp enough to show Theorem~\ref{maintheorem}.
So in Section~\ref{sec:gallery} we introduce the gallery with respect to an edge which combines a $(4,2)$-pillar with $(4,1)$-pillars
and finish the proof of Theorem \ref{maintheorem}.

\section{Preliminaries}

\subsection{Matrices}
We denote the eigenvalues of a real symmetric matrix $M$ of order $n$
by $\eta_1(M)\geqslant\eta_2(M)\geqslant\cdots\geqslant\eta_n(M)$. The largest (resp.\ smallest) eigenvalue of $M$ is also denoted by $\rho(M)$ (resp.\ $\eta_{\min}(M)$). The largest eigenvalue of $M$ is also called the spectral radius of $M$. The rank of $M$ is denoted by ${\rm rk}(M)$.

For a real symmetric $n\times n$ matrix $B$ and a real symmetric $m\times m$ matrix $C$ with $n>m$, we say that the eigenvalues of $C$ interlace the eigenvalues of $B$, if $\eta_{n-m+i}(B)\leqslant\eta_i(C)\leqslant\eta_i(B)$ for each $i=1,\ldots,m$. The following result is a special case of interlacing.

\begin{theorem}{\rm (Cf.~\cite[Theorem 9.1.1]{godsil2013})}\label{interlacing}
Let $B$ be a real symmetric $n\times n$ matrix and $C$ be a principal submatrix of $B$ of order $m$, where $m<n$. Then the eigenvalues of $C$ interlace the eigenvalues of $B$.
\end{theorem}

\subsection{Graphs}
A graph $G$ is an ordered pair $(V(G),E(G))$, where $V(G)$ is a finite set and $\displaystyle E(G)\subseteq \binom{V(G)}{2}$. The set $V(G)$ (resp.\ $E(G)$) is called the vertex set (resp.\ edge set) of $G$, and the cardinality of $V(G)$ (resp.\ $E(G)$) is called the order (resp.\ size) of $G$ and is denoted by $n_G$ (resp.\ $\varepsilon_G$). The adjacency matrix of $G$, denoted by $A(G)$, is a symmetric $(0,1)$-matrix indexed by $V(G)$, such that $(A(G))_{xy}=1$ if and only if $xy$ is an edge in $G$. The eigenvalues of $G$ are the eigenvalues of $A(G)$, and the spectral radius of $G$ is denoted by $\rho(G)$. The cardinality of a maximum independent set (resp.\ clique) in $G$ is called the independence number (resp.\ clique number) of $G$, denoted by $\alpha(G)$ (resp.\ $\omega(G)$).

The disjoint union of the graphs $G_1$ and $G_2$ is denoted by $G_1\dot\cup G_2$. For a graph $G$ and a subset $U \subseteq V(G)$, we denote by $G_U$ the subgraph of $G$ induced on $U$, i.e. $V(G_U)=U$ and $\displaystyle E(G_U)=E(G)\cap \binom{U}{2}$. For $H$ an induced subgraph of $G$, we denote by $N_G(H)$ the subgraph of $G$ induced on the vertices that have a neighbour in $H$ but are not in $H$, and we denote by $R_G(H)$ the subgraph induced on the vertices of $G$ that are neither in $H$ nor have a neighbour in $H$. If the graph $G$ is clear from the context, we will simply use $N(H)$ and $R(H)$.

Let $G$ be a graph.  We say $G$ is $k$-regular if the valency of every vertex in $G$ is a non-negative constant integer $k$. A graph $G$ of order $n$ is said to be strongly regular with parameters $(n,k,\lambda,\mu)$, if it is $k$-regular, every pair of adjacent vertices has $\lambda$ common neighbours, and every pair of distinct nonadjacent vertices has $\mu$ common neighbours. The following lemma is well-known {\rm (cf.~\cite[Section 10.1 and 10.2]{godsil2013})}.
\begin{lemma}\label{SRG}
Let $G$ be an $(n,k,\lambda,\mu)$ strongly regular graph with $k>\mu$. Then $G$ has exactly three distinct eigenvalues $k>\theta>\tau$ satisfying
\begin{align*}
\theta & =\frac{(\lambda-\mu)+\sqrt{(\lambda-\mu)^2+4(k-\mu)}}{2},\\
\tau & =\frac{(\lambda-\mu)-\sqrt{(\lambda-\mu)^2+4(k-\mu)}}{2}.
\end{align*}
Moreover, the multiplicity $m_\theta$ of $\theta$ is given by $\displaystyle m_{\theta}=-\frac{(n-1)\tau+k}{\theta-\tau}$.
\end{lemma}

\subsection{Seidel matrices}

Recall that a Seidel matrix $S$ of order $n$ is a symmetric $(0, \pm1)$-matrix
with $0$ on the diagonal and $\pm1$ otherwise.
The graph $G=G(S)$ corresponding to a Seidel matrix $S$ is the graph on $\{1,\ldots,n\}$ such that two distinct vertices $i$ and $j$ are adjacent if and only if $S_{ij}=-1$.
It follows immediately that $A(G) = \frac{1}{2} (\mathbf{J} - \mathbf{I} - S)$, where $\mathbf{J}$ is the all-ones matrix and $\mathbf{I}$ is the identity matrix.
Conversely, the Seidel matrix $S=S(G)$ corresponding to a graph $G$ can be obtained by $S=\mathbf{J-I}-2A(G)$.

Let $U\subseteq \{1,\ldots,n\}$. Define the diagonal matrix $D_U$ by $(D_U)_{ii}=1$ if $i\in U$ and $(D_U)_{ii}=-1$ if $i\notin U$. For a Seidel matrix $S$ we define the Seidel matrix $S_{sw}(U)$ by $S_{sw}(U):=D_USD_U$. For a graph $G$ with Seidel matrix $S(G)$ we denote by $G_{sw}(U)$ the graph $G((S(G))_{sw}(U))$. In other words, the graph $G_{sw}(U)$ is obtained from $G$ by switching with respect to $U$. If $G$ and $H$ are switching equivalent, then $S(G)$ and $S(H)$ are similar and hence have the same spectrum. The collection of graphs that can be obtained from $G$ by switching is called the switching class of $G$, denoted by $[G]$. For a Seidel matrix $S$, we define $[S]$ as $[S]:=[G]$, where $G$ is the corresponding graph of $S$. We call $[S]$ the switching class of $S$.

Let $S$ be a Seidel matrix
of order $n\geqslant2$. Let $S':=(\begin{smallmatrix}S & -S+\mathbf{I} \\  -S+\mathbf{I} & S\end{smallmatrix})$.
The graph $\mathcal{Sw}(S) := G(S')$ is called the switching graph of $S$. Note that $\mathcal{Sw}(S)$ only depends on $[S]$, that is, $\mathcal{Sw}(S_1)\cong\mathcal{Sw}(S_2)$ if and only if $S_1$ and $S_2$ are switching equivalent.
We define the independence number (resp. clique number) of $[S]$ as $\alpha([S]):=\alpha(\mathcal{Sw}(S))$ (resp. $\omega([S]):=\omega(\mathcal{Sw}(S))$). Note that $\alpha(\mathcal{Sw}(S))\geqslant 2$ and $\alpha(\mathcal{Sw}(S))=2$ if and only if $[S]=[\mathbf{I-J}]$. Similarly, $\omega(\mathcal{Sw}(S))\geqslant 2$ and $\omega(\mathcal{Sw}(S))=2$ if and only if $[S]=[\mathbf{J-I}]$.

\subsection{Some bounds on the smallest eigenvalue}
Let $G$ be a graph. From now on, we will use $\lambda_i$ to denote the eigenvalues of the Seidel matrix $S(G)$, and by $\theta_i$ to denote the eigenvalues of the adjacency matrix $A(G)$.

\begin{lemma}\label{SandA}
Let $S$ be a Seidel matrix, and $G$ be its corresponding graph. Then,
\begin{enumerate}[(i)]
  \item $\lambda_{\min}(S)\geqslant-2\rho(G)-1$;
  \item For any induced subgraph $H$ of $G$, we have  $\lambda_{\min}(S(H))\geqslant\lambda_{\min}(S)$.
\end{enumerate}
\end{lemma}
\begin{proof}
The first item follows immediately from the fact that $S=\mathbf{J-I}-2A(G)$. The second item is an easy consequence of Theorem \ref{interlacing}.
\end{proof}

Note that the Perron-Frobenius Theorem {\rm (cf.~\cite[Theorem 3.1.1]{brouwer1989distance})} implies that the spectral radius $\rho(G)$ of a connected graph $G$ is simple, and we can take an eigenvector for $\rho(G)$ with positive entries only. This means that, for any graph $G$, there exists an eigenvector $\mathbf{v}$ for the eigenvalue $\rho(G)$ with non-negative entries only.

\begin{lemma}\label{onelarge}
Let $S$ be a Seidel matrix with the smallest eigenvalue $\lambda_{\min}$. Let $G$ be its corresponding graph of $S$ with adjacency matrix $A$ and spectral radius $\displaystyle \rho~(= \frac{-\lambda_{\min}-1}{2})$.  Assume that $\mathbf{v}$ is an eigenvector of A with eigenvalue $\rho$, that is, $A\mathbf{v} = \rho \mathbf{v}$, and that $\mathbf v$ is not perpendicular to the all-ones vector $\mathbf j$. If there exists another eigenvector $\mathbf{w}$ of $A$ not perpendicular to the all-ones vector $\mathbf j$, say with eigenvalue $\theta\neq\rho$, then $\theta < \dfrac{-\lambda_{\min}-1}{2}$.
\end{lemma}
\begin{proof}
We denote by $\mathbf{U}$ the $2$-dimensional space spanned by $\mathbf{v}$ and $\mathbf{w}$. Then there exists a non-zero vector $\mathbf{u}\in \mathbf{U}-\{\textbf{0}\}$ such that $\mathbf{u\perp j}$. We find $S\textbf{u}\leqslant(-1-2\theta)\textbf{u}$, but equality would imply that $\rho=\theta$, as both $\mathbf{v}$ and $\mathbf{w}$ are not perpendicular to $\mathbf{j}$. Hence, $\lambda_{\min}<-1-2\theta$ and the conclusion holds.
\end{proof}

Lemma~\ref{onelarge} immediately implies the following proposition.

\begin{proposition}
If $\displaystyle\rho(G) > \frac{-\lambda_{\min}-1}{2}$, then any eigenvector for the eigenvalue $\displaystyle\frac{-\lambda_{\min}-1}{2}$ is perpendicular to $\mathbf{j}$.
\end{proposition}

The next proposition says that there exists at most one connected component of a graph $G$ whose spectral radius is larger than $\dfrac{-\lambda_{\min}(S(G))-1}{2}$.

\begin{proposition}\label{prop:one-large}
Let $S$ be a Seidel matrix with the smallest eigenvalue $\lambda_{\min}$. Let $G$ be its corresponding graph of $S$. Let $H$ be an induced subgraph of $G$. Let $R(H)$ be the subgraph of $G$ induced by the vertices which are neither in $H$ nor are adjacent to any vertex in $H$, that is, $V(R(H))=\{x\in V(G)\mid x \notin V(H), x\nsim y, \,\forall\, y\in V(H) \}$. If $\rho(H)>\dfrac{-\lambda_{\min}-1}{2}$, then $\rho(R(H))<\dfrac{-\lambda_{\min}-1}{2}$.
\end{proposition}
\begin{proof}
Let $H'$ be the disjoint union of $H$ and $R(H)$. Then the adjacency matrix of $H'$ is a diagonal block matrix $A'$ with two blocks, namely, the adjacency matrix $A(H)$ and $A(R(H))$ of $H$ and $R(H)$, respectively. Let $\mathbf{u}$ (resp.\ $\mathbf{v}$) be an non-negative eigenvector for $\rho(H)$ (resp.\ $\rho(R(H))$), that is, $A(H)\mathbf{u}=\rho(H)\mathbf{u}$ and $A(R(H))\mathbf{v}=\rho(R(H))\mathbf{v}$.

Define $\mathbf{w}$ by
\begin{equation}\nonumber
\mathbf{w}_x=
\begin{cases}
\mathbf{u}_x,& \text{if $x \in V(H)$},\\
0,& \text{if $x\in V(R(H))$},
\end{cases}
\end{equation}
and, in similar fashion, define $\mathbf{w'}$ from $\mathbf{v}$. Note that $A'\mathbf{w}=\rho(H)\mathbf{w}$, $A'\mathbf{w'}=\rho(R(H))\mathbf{w'}$, $\mathbf{w\not\perp j}$, $\mathbf{w'\not\perp j}$ and $\mathbf{w\perp w'}$.

Let $S'$ be the Seidel matrix of $H'$. By Theorem \ref{interlacing}, we have $\lambda_{\min}(S')\geqslant\lambda_{\min}$. So, if $\rho(H)>\dfrac{-\lambda_{\min}-1}{2}$, then $\rho(R(H))<\dfrac{-\lambda_{\min}-1}{2}$ by Lemma \ref{onelarge}. This shows the proposition.
\end{proof}

Let $M$ be a symmetric $n\times n$ matrix and $\pi:=\{V_1,\ldots,V_r\}$ be a partition of $\{1,\ldots,n\}$. Let $M_{ij}$ be the submatrix of $M$ whose rows are indexed by $V_i$ and whose columns are indexed by $V_j$. We say $\pi$ is an equitable partition with respect to $M$ if $M_{ij}$ has constant row sum for all $1\leqslant i,j\leqslant r$. For an equitable partition $\pi$ with respect to $M$, let $q_{ij}$ be the row sum of $M_{ij}$, for $1\leqslant i,j\leqslant r$. The quotient matrix $Q$ of $M$ with respect to $\pi$ is defined as $Q=(q_{ij})_{1\leqslant i,j\leqslant r}$.

\begin{lemma}\label{quotient}
Let $M$ be a symmetric $n\times n$ matrix. If $\pi$ is an equitable partition of $M$ and $Q$ is the quotient matrix with respect to $\pi$ of $M$, then every eigenvalue of $Q$ is an eigenvalue of $M$.
\end{lemma}
\begin{proof}
Let $\pi:=\{V_1,\ldots,V_r\}$ be a equitable partition of $M$. Let $\lambda$ be an eigenvalue of $Q$ and $\mathbf{v}$ be an eigenvector of $Q$ with $\lambda$. Let $\mathbf{w}$ be the vector in $\mathbb{R}^n$ such that $\mathbf{w}_k=\mathbf{v}_i$ for $k\in V_i$, $i\in\{1,\ldots,r\}$.

Let $k\in V_i$. Then $\sum\limits_{l\in V_j}M_{kl}=q_{ij}$. It follows that $(M\mathbf{w})_k=\sum\limits_{l=1}^{n}M_{kl}\mathbf{w}_l=\sum\limits_{j=1}^{r}q_{ij}\mathbf{v}_j=\lambda\mathbf{v}_i=\lambda\mathbf{w}_k$. This shows the lemma.
\end{proof}

If $M$ is the adjacency matrix of a graph $G$, and $\pi$ is an equitable partition of $\{1,\ldots,n\}$ with respect to $M$, then we say that $\pi$ is an equitable partition of $G$. Note that in this case $\pi$ is also an equitable partition with respect to the Seidel matrix of $G$.

\begin{corollary}\label{S-quotient}
If $\pi$ is an equitable partition of a graph $G$ and $Q$ is the quotient matrix with respect to the Seidel matrix $S$ of $G$, then every eigenvalue of $Q$ is an eigenvalue of $S$.
\end{corollary}

\subsection{Smith's Theorem}
Now we present Smith's Theorem in the year of $1970$, in which Smith determined all graphs with spectral radius $2$. Note that the corresponding Seidel matrices of these graphs have their smallest eigenvalues at least $-5$.

\begin{theorem}{\rm (Cf.~\cite[Section 3.2]{brouwer1989distance})}\label{spectralradius2}
The only connected graphs having spectral radius $2$ are the following graphs (the number of vertices is one more than the index given).

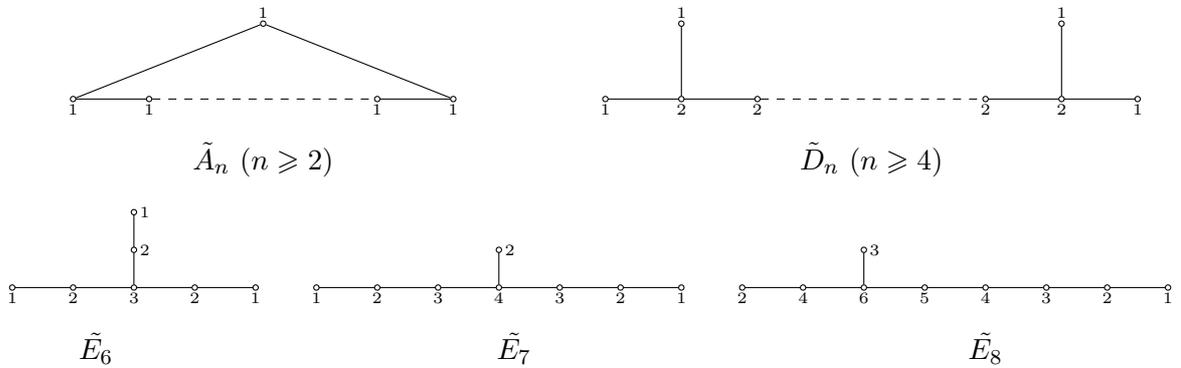
\begin{figure}[ht]
    \centering
    \begin{tikzpicture}
    \draw (1.5,1.7) node {$\tilde{A}_n\ (n\geqslant2)$};
    \draw (9.5,1.7) node {$\tilde{D}_n\ (n\geqslant4)$};
    \draw (-0.7,-0.8) node {$\tilde{E_6}$};
    \draw (4.8,-0.8) node {$\tilde{E_7}$};
    \draw (11,-0.8) node {$\tilde{E_8}$};
    \tikzstyle{every node}=[draw,circle,fill=white,minimum size=2pt,
                            inner sep=0pt]
                            {every label}=[\tiny]
    \draw (1.5,3.5) node (0) [label=above:$1$] {};
    \draw (-1,2.5) node (1) [label=below:$1$] {}
        -- ++(0:1cm) node (2) [label=below:$1$] {};
    \draw (4,2.5) node (4) [label=below:$1$] {}
        -- ++(180:1cm) node (3) [label=below:$1$] {};
    \draw [dashed] (2) -- (3);
    \draw (0) -- (1);
    \draw (0) -- (4);

    \draw (7,3.5) node (01) [label=above:$1$] {};
    \draw (12,3.5) node (02) [label=above:$1$] {};
    \draw (6,2.5) node (11) [label=below:$1$] {}
        -- ++(0:1cm) node (12) [label=below:$2$] {}
        -- ++(0:1cm) node (13) [label=below:$2$] {};
    \draw (11,2.5) node (14) [label=below:$2$] {}
        -- ++(0:1cm) node (15) [label=below:$2$] {}
        -- ++(0:1cm) node (16) [label=below:$1$] {};
    \draw [dashed] (13) -- (14);
    \draw (01) -- (12);
    \draw (02) -- (15);

   \draw (-1.8,0) node (1b) [label=below:$1$] {}
       -- ++(0:0.8cm) node (2b) [label=below:$2$] {}
       -- ++(0:0.8cm) node (3b) [label=below:$3$] {}
       -- ++(0:0.8cm) node (4b) [label=below:$2$] {}
       -- ++(0:0.8cm) node (5b) [label=below:$1$] {};
   \draw (3b)
       -- ++(90:0.5cm) node (6b) [label=right:$2$] {}
       -- ++(90:0.5cm) node (7b) [label=right:$1$] {};

   \draw (2.2,0) node (1c) [label=below:$1$] {}
       -- ++(0:0.8cm) node (2c) [label=below:$2$] {}
       -- ++(0:0.8cm) node (3c) [label=below:$3$] {}
       -- ++(0:0.8cm) node (4c) [label=below:$4$] {}
       -- ++(0:0.8cm) node (5c) [label=below:$3$] {}
       -- ++(0:0.8cm) node (6c) [label=below:$2$] {}
       -- ++(0:0.8cm) node (7c) [label=below:$1$] {};
   \draw (4c)
       -- ++(90:0.5cm) node (8c) [label=right:$2$] {};

   \draw (7.8,0) node (1d) [label=below:$2$] {}
       -- ++(0:0.8cm) node (2d) [label=below:$4$] {}
       -- ++(0:0.8cm) node (3d) [label=below:$6$] {}
       -- ++(0:0.8cm) node (4d) [label=below:$5$] {}
       -- ++(0:0.8cm) node (5d) [label=below:$4$] {}
       -- ++(0:0.8cm) node (6d) [label=below:$3$] {}
       -- ++(0:0.8cm) node (7d) [label=below:$2$] {}
       -- ++(0:0.8cm) node (8d) [label=below:$1$] {};
   \draw (3d)
       -- ++(90:0.5cm) node (9d) [label=right:$3$] {};
    \end{tikzpicture}
\end{figure}

For each graph, the corresponding eigenvector is indicated by the integers at the vertices. Moreover, each connected graph with spectral radius less than $2$ is a subgraph of the above graphs, and each connected graph with spectral radius greater than $2$ contains one of these graphs.
\end{theorem}

\begin{remark}
{\rm This theorem shows that each graph with spectral radius less than $2$ is a forest. }
\end{remark}

As an easy consequence of Theorem \ref{spectralradius2}, we determine the minimal graphs with spectral radius larger than $2$, that is, the graphs with spectral radius larger than $2$ such that any proper induced subgraph has spectral radius at most $2$.

\begin{corollary}\label{>2}
The minimal graphs with spectral radius greater than $2$ are the $18$ graphs listed in Figure~$\ref{fig:>2}$.

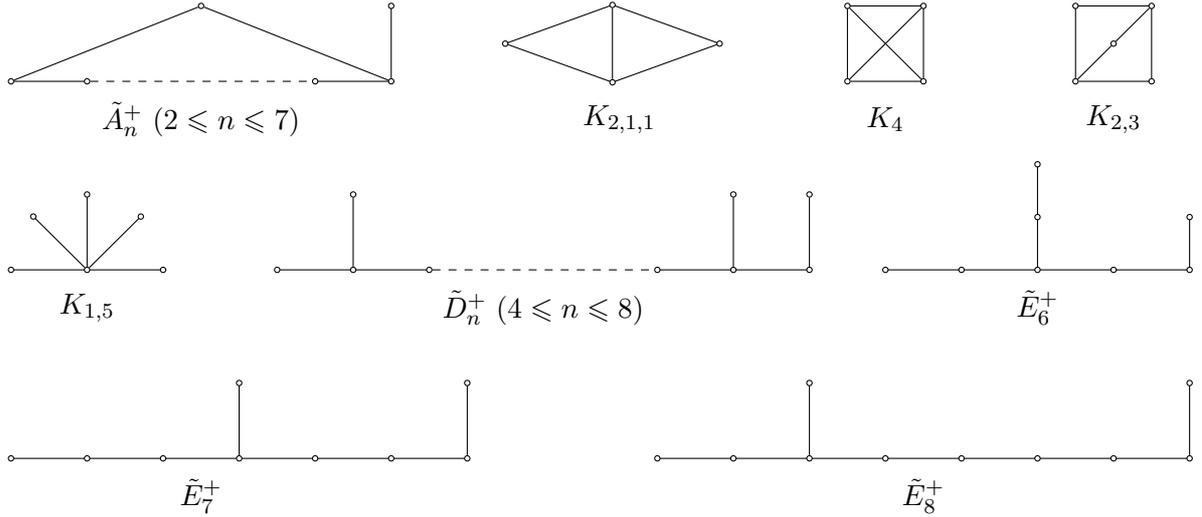
\begin{figure}[ht]
    \centering
    \begin{tikzpicture}
    \draw (2,3.5) node {$\tilde{A}^+_n\ (2\leqslant n\leqslant7)$};
    \draw (7.5,3.5) node {$K_{2,1,1}$};
    \draw (11,3.5) node {$K_4$};
    \draw (14,3.5) node {$K_{2,3}$};
    \draw (0.5,1) node {$K_{1,5}$};
    \draw (6.5,1) node {$\tilde{D}^+_n\ (4\leqslant n\leqslant8)$};
    \draw (13,1) node {$\tilde{E}^+_6$};
    \draw (2,-1.5) node {$\tilde{E}^+_7$};
    \draw (11.5,-1.5) node {$\tilde{E}^+_8$};
    \tikzstyle{every node}=[draw,circle,fill=white,minimum size=2pt,
                            inner sep=0pt]
                            {every label}=[\tiny]
    \draw (2,5) node (0) [label=above:$ $] {};
    \draw (-0.5,4) node (1) [label=below:$ $] {}
        -- ++(0:1cm) node (2) [label=below:$ $] {};
    \draw (4.5,4) node (4) [label=below:$ $] {}
        -- ++(180:1cm) node (3) [label=below:$ $] {};
    \draw [dashed] (2) -- (3);
    \draw (0) -- (1);
    \draw (0) -- (4);
    \draw (4)
        -- ++(90:1cm) node (01) [label=above:$ $] {};

    \draw (6,4.5) node (0a) [label=left:$ $] {}
         -- ++(20:1.5cm) node (0b) [label=above:$ $] {}
         -- ++(-20:1.5cm) node (0c) [label=right:$ $] {}
         -- ++(200:1.5cm) node (0d) [label=below:$ $] {}
         --(0a);
    \draw (0b) -- (0d);

    \draw (10.5,4) node (0e) [label=left:$ $] {}
         -- ++(90:1cm) node (0f) [label=left:$ $] {}
         -- ++(0:1cm) node (0g) [label=right:$ $] {}
         -- ++(270:1cm) node (0h) [label=right:$ $] {}
         --(0e);
    \draw (0e) -- (0g);
    \draw (0f) -- (0h);

    \draw (13.5,4) node (0e) [label=left:$ $] {}
         -- ++(90:1cm) node (0f) [label=left:$ $] {}
         -- ++(0:1cm) node (0g) [label=right:$ $] {}
         -- ++(270:1cm) node (0h) [label=right:$ $] {}
         --(0e);
    \draw (14,4.5) node (00) [label=left:$ $] {};
    \draw (0e) -- (00) -- (0g);

    \draw (-0.5,1.5) node (0i) [label=left:$ $] {}
         -- ++(0:1cm) node (0j) [label=left:$ $] {}
         -- ++(0:1cm) node (0k) [label=right:$ $] {};
    \draw (0j)
         -- ++(45:1cm) node (0l) [label=right:$ $] {};
    \draw (0j)
         -- ++(90:1cm) node (0m) [label=right:$ $] {};
    \draw (0j)
         -- ++(135:1cm) node (0n) [label=right:$ $] {};

    \draw (4,2.5) node (01) [label=above:$ $] {};
    \draw (9,2.5) node (02) [label=above:$ $] {};
    \draw (10,2.5) node (03) [label=above:$ $] {};
    \draw (3,1.5) node (11) [label=below:$ $] {}
        -- ++(0:1cm) node (12) [label=below:$ $] {}
        -- ++(0:1cm) node (13) [label=below:$ $] {};
    \draw (8,1.5) node (14) [label=below:$ $] {}
        -- ++(0:1cm) node (15) [label=below:$ $] {}
        -- ++(0:1cm) node (16) [label=below:$ $] {};
    \draw [dashed] (13) -- (14);
    \draw (01) -- (12);
    \draw (02) -- (15);
    \draw (16) -- (03);

   \draw (11,1.5) node (1b) [label=above:$ $] {}
       -- ++(0:1cm) node (2b) [label=above:$ $] {}
       -- ++(0:1cm) node (3b) [label=above:$ $] {}
       -- ++(0:1cm) node (4b) [label=above:$ $] {}
       -- ++(0:1cm) node (5b) [label=above:$ $] {};
   \draw (3b)
       -- ++(90:0.7cm) node (6b) [label=right:$ $] {}
       -- ++(90:0.7cm) node (7b) [label=right:$ $] {};
   \draw (5b)
       -- ++(90:0.7cm) node (8b) [label=right:$ $] {};

   \draw (-0.5,-1) node (1c) [label=above:$ $] {}
       -- ++(0:1cm) node (2c) [label=above:$ $] {}
       -- ++(0:1cm) node (3c) [label=above:$ $] {}
       -- ++(0:1cm) node (4c) [label=above:$ $] {}
       -- ++(0:1cm) node (5c) [label=above:$ $] {}
       -- ++(0:1cm) node (6c) [label=above:$ $] {}
       -- ++(0:1cm) node (7c) [label=above:$ $] {};
   \draw (4c)
       -- ++(90:1cm) node (8c) [label=right:$ $] {};
   \draw (7c)
       -- ++(90:1cm) node (8b) [label=right:$ $] {};

   \draw (8,-1) node (1d) [label=above:$ $] {}
       -- ++(0:1cm) node (2d) [label=above:$ $] {}
       -- ++(0:1cm) node (3d) [label=above:$ $] {}
       -- ++(0:1cm) node (4d) [label=above:$ $] {}
       -- ++(0:1cm) node (5d) [label=above:$ $] {}
       -- ++(0:1cm) node (6d) [label=above:$ $] {}
       -- ++(0:1cm) node (7d) [label=above:$ $] {}
       -- ++(0:1cm) node (8d) [label=above:$ $] {};
   \draw (3d)
       -- ++(90:1cm) node (9d) [label=right:$ $] {};
   \draw (8d)
       -- ++(90:1cm) node (8b) [label=right:$ $] {};
   \end{tikzpicture}
   \caption{The 18 minimal graphs with spectral radius greater than $2$}
    \label{fig:>2}
\end{figure}
\end{corollary}

\section{Forbidden subgraphs}

For $\lambda < 0$, let $\mathcal{F}_{\lambda}$ denote the set of minimal forbidden graphs for the smallest Seidel eigenvalue $\lambda$, that is,
\begin{equation*}
\mathcal{F}_{\lambda}:=\{G\mid\lambda_{\min}(S(G))<\lambda \text{~and any induced proper subgraph $H$ of $G$ satisfies $\lambda_{\min}(S(H))\geqslant\lambda$}\}.
\end{equation*}
Jiang and Polyanskii \cite[Theorem~1]{jiang2017forbidden} showed that the set $\mathcal{F}_{-5}$ is finite.
Since we are talking about Seidel eigenvalues, only the switching classes of such graphs are needed. Now we determine some graphs inside $\mathcal{F}_{-5}$. In order to do so, we define the following.
For a graph $G$, let $G(s,t)$ be the disjoint union of $G$, $s$ isolated vertices, and $t$ copies of $K_2$, where $s,t$ are non-negative integers.
In particular, we write $G(s)$ for $G(s,0)$.

Using the graphs of Corollary \ref{>2}, we obtain the following lemma.

\begin{lemma}\label{forbidden}
Table $\ref{18graphs}$ gives $18$ graphs that belong to $\mathcal{F}_{-5}$.
\begin{table}[h]
\begin{center}
\begin{tabular}{c  c  c  c  c  c  c  c c c}\hline
   $G$ & $\tilde{A}^+_2$ & $\tilde{A}^+_3$  & $\tilde{A}^+_4$ &  $\tilde{A}^+_5$  & $\tilde{A}^+_6$  & $\tilde{A}^+_7$  &  $K_{2,1,1}$ & $K_4$ & $K_{2,3}$\\[{0.5em}]\hline
   $s$ such that $G(s)\in\mathcal{F}_{-5}$ & 41 & 67 & 97 & 130 & 165 & 201 & 11 & 5 & $19$\\[{0.5em}] \hline
    $G$ & $K_{1,5}$ & $\tilde{D}^+_4$ & $\tilde{D}^+_5$ & $\tilde{D}^+_6$ & $\tilde{D}^+_7$ & $\tilde{D}^+_8$ & $\tilde{E}^+_6$& $\tilde{E}^+_7$ & $\tilde{E}^+_8$ \\[{0.5em}]\hline
 $s$ such that $G(s)\in\mathcal{F}_{-5}$ & $41$ & 137 & 225 & 327 & 439 & 557 & 465 & 966 & 2477\\[{0.5em}] \hline
\end{tabular}
\end{center}
\caption{$18$ minimal forbidden graphs for $\lambda=-5$.}\label{18graphs}
\end{table}
\end{lemma}
\begin{proof}
For each of the graphs $G$ of Corollary \ref{>2}, we determine the smallest integer $s$ that satisfies $\lambda_{\min}(S(G(s))) <-5$. That the obtained graphs $G(s)$ belong to $\mathcal{F}_{-5}$, follows from the fact that any such $G$ is a minimal graph with spectral radius larger than $2$ (by Corollary \ref{>2}) and Lemma \ref{SandA}.
\end{proof}

\begin{remark}
This lemma shows that there exists a graph of order $2487$ inside $\mathcal{F}_{-5}$. We do not know whether this graph has the largest order inside $\mathcal{F}_{-5}$. Nevertheless this suggests that it may be difficult to find all graphs in $\mathcal{F}_{-5}$. This explains the lower bound of Neumaier's claim in the introduction.
\end{remark}

The following lemma is of crucial importance for this paper.
\begin{lemma}\label{36}
Let $r\geqslant 2,s\geqslant0,t\geqslant0$ be integers such that $s+t\geqslant1$. Then, the following hold.
\begin{enumerate}[(i)]
  \item The smallest eigenvalue of $S(K_{1,r}(s,t))$ satisfies $\lambda_{\min}(S(K_{1,r}(s,t)))\geqslant-5$ if and only if $(r-4)(s+4t-4)\leqslant36$;
  \item Given a graph $G$ with $r$ vertices, let $C(G)$ be the cone of $G$, that is, adding a new vertex to $G$ and joining it with all vertices of $G$. Then $\lambda_{\min}(S(C(G)(s,t)))\leqslant\lambda_{\min}(S(K_{1,r}(s,t)))$.
\end{enumerate}
\end{lemma}
\begin{proof}
$(i)$
First, we consider the case when $r\geqslant2$, $s\geqslant1$ and $t\geqslant1$. Let $v$ be the vertex of valency $r$ in $K_{1,r}$, $V_1=V(K_{1,r}-{v})$, $V_2=V(\overline{K}_s)$ and $V_3=V(tK_2)$. Consider a partition $\pi=\{\{v\}, V_1, V_2, V_3\}$ of $K_{1,r}(s,t)$. The partition $\pi$ is equitable with quotient matrix $Q$ with respect to $S(K_{1,r}(s,t))$:
\begin{gather*}
Q=\begin{pmatrix}
0 & -r & s & 2t \\ -1 & r-1 & s & 2t \\ 1 & r & s-1 & 2t \\ 1 & r & s & 2t-3
\end{pmatrix}.
\end{gather*}

Note that $\det(Q+3\mathbf{I})=-16t(r-1)$. As $r\geqslant2$ and $t\geqslant1$, we see that $\lambda_{\min}(Q)<-3$. By Theorem \ref{interlacing}, we observe that $S(K_{1,r}(s,t))$ has at most one eigenvalue at most $-3$, as $\lambda_{\min}(S(\overline{K_{r+s}}$ $\dot\cup$ $tK_2))=-3$. This implies that $\lambda_{\min}(Q)=\lambda_{\min}(S(K_{1,r}(s,t)))$, by Lemma \ref{quotient}. Next, we find that
\begin{equation*}
\det(Q+5\mathbf{I})=-8((r-4)(s+4t-4)-36).
\end{equation*}
This shows that $(i)$ is correct, if $r\geqslant2$, $s\geqslant1$ and $t\geqslant1$.

If $s=0$ or $t=0$, then with a similar argument we see that $(i)$ is true.

$(ii)$ Fix $a,b,c\in\mathbb{R}$ such that $Q\begin{pmatrix}1\\ a \\ b \\ c\end{pmatrix}=\lambda_{\min}(Q)\begin{pmatrix}1\\ a \\ b \\c\end{pmatrix}$.

We find

\begin{align*}
  \lambda_{\min}(Q) & =-ra+sb+2tc \\
  a\lambda_{\min}(Q) & =-1+(r-1)a+sb+2tc.
\end{align*}
If $a\leqslant0$, then
\begin{equation*}
a\lambda_{\min}(Q)+1\leqslant sb+2tc\leqslant \lambda_{\min}(Q).
\end{equation*}
This gives a contradiction, as $\lambda_{\min}(Q)<0$ and $a\lambda_{\min}(Q)+1>0$. It follows that $a>0$.

Let $\{w\}=V(C(G))-V(G)$, $W_1=V(G)$, $W_2=V(\overline{K}_s)$ and $W_3=V(tK_2)$. Let $\mathbf{w}$ be a vector in $\mathbb{R}^{V(C(G)(s,t))}$ such that
\begin{equation}\nonumber
\mathbf{w}_x=\left\{
 \begin{array}{ll}
1,& \text{if $x=w$},\\
a,& \text{if $x\in W_1$},\\
b,& \text{if $x\in W_2$},\\
c,& \text{if $x\in W_3$}.
 \end{array}
 \right.
\end{equation}

For any vertex $x$ in $V(C(G)(s,t))$, note that $(S(C(G)(s,t))\mathbf{w})_x\leqslant(S(K_{1,r}(s,t))\mathbf{w})_x=(\lambda_{\min}(Q)\mathbf{w})_x$. This implies that $\lambda_{\min}(S(C(G)(s,t)))\leqslant\lambda_{\min}(Q)$. It shows $(ii)$.
\end{proof}

Analogous computations show that the following graphs also belong to $\mathcal{F}_{-5}$. We omit the details here.

\begin{lemma}\label{forbidden2}
The graphs $B_1(14)$ and $B_2(9)$ belong to $\mathcal{F}_{-5}$, where the graphs $B_1$ and $B_2$ are listed in Figure~$\ref{figB}$.
\begin{figure}[ht]
\centering
\begin{tikzpicture}
    \draw (1,-0.5) node {$B_1$};
    \draw (6,-0.5) node {$B_2$};
    \tikzstyle{every node}=[draw,circle,fill=white,minimum size=2pt,
                            inner sep=0pt]
                            {every label}=[\tiny]

    \draw (0,0) node (11) [label=below:$ $] {}
        -- ++(90:1cm) node (12) [label=below:$ $] {}
        -- ++(-30:1cm) node (13) [label=below:$ $] {}
        -- ++(30:1cm) node (14) [label=below:$ $] {}
        -- ++(-90:1cm) node (15) [label=below:$ $] {};
    \draw  (13) -- (15);
    \draw  (11) -- (13);

    \draw (5,1.5) node (21) [label=below:$ $] {}
        -- ++(-30:1cm) node (22) [label=below:$ $] {}
        -- ++(210:1cm) node (23) [label=below:$ $] {}
        -- ++(-30:1cm) node (24) [label=below:$ $] {}
        -- ++(30:1cm) node (25) [label=below:$ $] {}
        --(22);
    \draw (22)
        -- ++(30:1cm) node (26) [label=below:$ $] {};
    \draw (22) -- (24);
\end{tikzpicture}
\caption{Two more forbidden graphs in $\mathcal{F}_{-5}$}
\label{figB}
\end{figure}
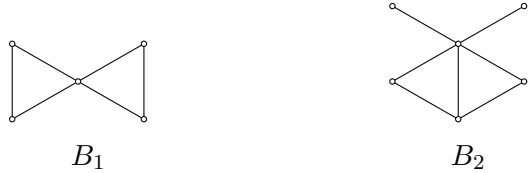
\end{lemma}

\section{The independence number is at least \texorpdfstring{$49$}{49}}\label{sec:49}
Let $S$ be a Seidel matrix with $\lambda_{\min}(S)=-5$ of order $n\geqslant277$. In this section, we show that, if the switching class of $S$ has independence number at least $49$, then {\rm rk}$(S+5\mathbf{I})\geqslant\frac{2n}{3}+1$. This shows that in this case Theorem \ref{maintheorem} is true.

We start with the small spectral radius.

\begin{proposition}\label{LS2}
Let $S$ be a Seidel matrix with $\lambda_{\min}(S)=-5$ of order $n$. If the switching class of $S$ contains a graph $G$ with spectral radius $\rho(G)\leqslant2$, then {\rm rk}$(S+5\mathbf{I})\geqslant\frac {2n}{3}+1$.
\end{proposition}
\begin{proof}
If $\rho(G)<2$, then $S(H')+5\mathbf{I}$ has full rank, by Lemma \ref{SandA}. Next we may assume $\rho(G)=2$. Clearly, $\lambda_{\min}(S)\geqslant-2\rho(G)-1=-5$, by Lemma \ref{SandA}. The multiplicity of $-5$ of $S(G)$ is one less than the number of connected components of $G$ with spectral radius $2$. As each connected component with spectral radius $2$ has at least $3$ vertices, it follows that {\rm rk}$(S+5\mathbf{I})\geqslant\frac {2n}{3}+1$. This shows the proposition.
\end{proof}

The next lemma gives a lower bound for the rank of $S+5\mathbf{I}$, where $S$ is a Seidel matrix.

\begin{lemma}\label{BoundedrkbyH}
Let $S$ be a Seidel matrix with $\lambda_{\min}(S)=-5$ of order $n$. Assume the switching class of $S$ contains a graph $G$ with $\rho(G)>2$. Let $d_{\max}$ be the maximum valency of $G$. Let $H$ be an induced subgraph of $G$ with $\rho(H)>2$. Let $n_H$ (resp.\ $\varepsilon_H$) be the order (resp.\ size) of $H$. Let $\alpha(H)$ be the independence number of $H$. Then ${\rm rk} (S+5\mathbf{I})\geqslant n-n_H(1+d_{\max})+2\varepsilon_H+\alpha(H)$.
\end{lemma}
\begin{proof}
For a vertex $x$ in $G$, denote by $d_x(G)$ the valency of $x$ in $G$. Let $d_{\max}(V(H))$ be the maximum valency among all vertices in $V(H)$ in $G$, that is, $d_{\max}(V(H)):=\max\{d_x\mid x\in V(H)\}$. Let $R(H)$ be the subgraph of $G$ induced on the vertices that are neither vertices of $H$ nor have a neighbour in $H$. Then $R(H)$ has at least $n-(n_H(1+d_{\max}(V(H)))-2\varepsilon_H) = n-n_H(1+d_{\max}(V(H)))+2\varepsilon_H$ vertices. As $\rho(R(H))<2$, by Lemma \ref{SandA}, it follows that $S(R(H)(\alpha(H)))+5\mathbf{I}$ has full rank. This shows that
\begin{align*}
  {\rm rk} (S+5\mathbf{I}) & \geqslant n-n_H(1+d_{\max}(V(H)))+2\varepsilon_H+\alpha(H) \\
   & \geqslant n-n_H(1+d_{\max})+2\varepsilon_H+\alpha(H).
\end{align*}
\end{proof}

As a consequence of Lemma \ref{BoundedrkbyH}, we have the following theorem.

\begin{theorem}\label{LS16}
Let $S$ be a Seidel matrix with $\lambda_{\min}(S)=-5$ of order $n\geqslant277$. If the switching class of $S$ contains a graph $G$ with maximum valency $d_{\max}\leqslant 16$, then ${\rm rk} (S+5\mathbf{I}) \geqslant \frac{2n}{3}+1$.
\end{theorem}
\begin{proof}
Let $G$ be a graph in the switching class of $S$ with $d_{\max}\leqslant16$. By Proposition \ref{LS2}, we may assume $\rho(G)>2$. For any vertex $x$ in $G$, we denote the valency of $x$ in $G$ by $d_x$. Let $H$ be a minimal induced subgraph of $G$ with $\rho(H)>2$. Let $d_{\max}(V(H))$ be the maximum valency of all vertices in $V(H)$ in $G$, that is, $d_{\max}(V(H)):=\max\{d_x\mid x\in V(H)\}$.
Clearly, $d_{\max}(V(H))\geqslant3$. If $d_{\max}(V(H))=3$, then, by Corollary \ref{>2}, we have $n_{H}\leqslant 10$, $\varepsilon_H\geqslant n_H-1$ and $\alpha(H)\geqslant 1$. By Lemma \ref{BoundedrkbyH}, we have
\begin{align*}
{\rm rk}(S+5\mathbf{I})&\geqslant n-n_H(1+d_{\max}(V(H)))+2\varepsilon_H+\alpha(H)\\
&\geqslant n-n_H(1+d_{\max}(V(H)))+2(n_H-1)+1\\
&\geqslant n-21\\
&>\frac{2n}{3}+1,
\end{align*}
as $n\geqslant277$. On the other hand, if $d_{\max}(V(H))\geqslant4$, then $n_H\leqslant 6$, $\varepsilon_H\geqslant n_H-1$ and $\alpha(H)\geqslant 1$, by Corollary \ref{>2}. By Lemma \ref{BoundedrkbyH}, we have
\begin{align*}
{\rm rk}(S+5\mathbf{I})&\geqslant n-n_H(1+d_{\max})+2\varepsilon_H+\alpha(H)\\
&\geqslant n-n_H(1+d_{\max})+2(n_H-1)+1\\
&\geqslant n-91\\
&>\frac{2n}{3}+1,
\end{align*}
as $n\geqslant277$. This shows the theorem.
\end{proof}

Now, we show the main result of this section.

\begin{theorem}\label{LS49}
Let $S$ be a Seidel matrix with $\lambda_{\min}(S)=-5$ of order $n\geqslant277$. If the independence number $\alpha([S])$ of $[S]$ satisfies $\alpha([S])\geqslant49$, then {\rm rk}$(S+5\mathbf{I})\geqslant\frac{2n}{3}+1$.
\end{theorem}
\begin{proof}
Let $\alpha:=\alpha([S])$. Take a graph $G$ in the switching class of $S$ with independence number $\alpha(G)=\alpha$. Let $C$ be an independent set of $G$ of order $\alpha$. We may assume that all vertices, that are not in $V(C)$, have at most $\lfloor\frac{\alpha}{2}\rfloor$ neighbours in $C$.

Let $x$ be a vertex outside $V(C)$ and assume that $x$ has $r$ neighbours in $C$. The subgraph of $G$ induced on $V(C)$ $\cup$ $\{x\}$ is isomorphic to $K_{1,r}(\alpha-r)$ with $\alpha-r\geqslant\lceil\frac{\alpha}{2}\rceil\geqslant25$, as $\alpha\geqslant49$.
As $\lambda_{\min}(S(K_{1,r}(s)))\geqslant-5$ if and only if $(r-4)(s-4)\leqslant36$, by Lemma \ref{36} $(i)$, we see $r\leqslant4$ when $r + s = \alpha \geqslant 49$.
That is, every vertex $x$ outside $C$ has at most $4$ neighbours in $C$.

Now we show the following claim.

\begin{claim}\label{dG6}
The maximum valency $d_{\max}$ of $G$ is at most $6$.
\end{claim}
\noindent{\bf Proof of Claim \ref{dG6}:} Assume $d_{\max}\geqslant7$. Let $x$ be a vertex with valency at least $7$, and let $y_1,\ldots,y_7$ be $7$ of its neighbours. Let $H$ be the subgraph of $G$ induced by $\{x, y_1, \ldots, y_7\}$. The number of vertices in $C$ that are in $H$ or have at least one neighbour in $H$ is at most $4\times8=32$. Note that $\lambda_{\min}(S(H(17)))\leqslant\lambda_{\min}(S(K_{1,7}(17)))<-5$, by Lemma \ref{36} $(i)$ and $(ii)$. This shows the claim.
\qed

Therefore, {\rm rk}$(S+5\mathbf{I})\geqslant\frac{2n}{3}+1$, by Theorem \ref{LS16}.
\end{proof}

Note that, by Theorem~\ref{LS49} and the Ramsey theory, it follows that Conjecture~\ref{conj:LS} is true when $n$ is sufficiently large, a result also obtained by Neumaier \cite{neumaier1989graph}.

\section{The switching class contains a triangle-free graph}\label{sec:triangle-free}
Let $S$ be a Seidel matrix with $\lambda_{\min}(S)=-5$ of order $n\geqslant277$. In this section, we will show that Theorem~\ref{maintheorem} is true when the switching class contains a triangle-free graph.
Our main result of this section is as follows.

\begin{theorem}\label{LStrianglefree}
Let $S$ be a Seidel matrix with $\lambda_{\min}(S)=-5$ of order $n\geqslant277$. Assume that the switching class of $S$ contains a triangle-free graph $G$. Then {\rm rk}$(S+5\mathbf{I})\geqslant\frac{2n}{3}+1$.
\end{theorem}
\begin{proof}
Let $G$ be a triangle-free graph in $[S]$.
By Proposition \ref{LS2} and Theorem~\ref{LS16}, we may assume that $\rho(G)>2$ and $d_{\max}\geqslant17$.
Note that $\alpha([S])\geqslant 1+d_{\max}$, as $K_{1,t}$ is switching equivalent to an independent set of order $t+1$. Hence, if $d_{\max}\geqslant48$, then, by Theorem \ref{LS49}, {\rm rk}$(S+5\mathbf{I})\geqslant\frac{2n}{3}+1$.
So we only need to consider $17\leqslant d_{\max}\leqslant47$. For a subgraph $H$ of $G$, we denote by $n_H$ the order of $H$, and by $R(H)$ the subgraph of $G$ induced on the vertices of $G$ that are neither in $H$ nor have a neighbour in $H$.

\begin{claim}\label{K2,3}
The graph $G$ contains an induced subgraph isomorphic to $K_{2,3}$.
\end{claim}
\noindent{\bf Proof of Claim \ref{K2,3}:} Let $x$ be a vertex of $G$ with valency $d_x=d_{\max}\geqslant17$. We partition the neighbours of $x$ into $3$ sets, say $S_1, S_2$ and $S_3$, such that $|S_1|=|S_2|=5$. Let $H_i$ be the subgraph induced on $S_i\cup\{x\}$, for $i=1,2,3$. Then $R(H_i)$ is triangle-free and satisfies $\rho(R(H_i))<2$, by Proposition \ref{prop:one-large}. It follows that $n_{R(H_i)}\leqslant2\alpha(R(H_i))$, by Theorem \ref{spectralradius2}.

As $K_{1,5}(41)$ and $K_{1,7}(17)$ are in $\mathcal{F}_{-5}$, by Lemma \ref{36} $(i)$, we find
\begin{align*}
n_{R(H_1)}&\leqslant2\alpha(R(H_1))\leqslant2\times40=80,\\
n_{R(H_2)}&\leqslant2\alpha(R(H_2))\leqslant2\times40=80,\\
n_{R(H_3)}&\leqslant2\alpha(R(H_3))\leqslant2\times16=32.
\end{align*}

Then $1+|S_1|+|S_2|+|S_3|+n_{R(H_1)}+n_{R(H_2)}+n_{R(H_3)}\leqslant1+47+80+80+32=240$, as $|S_1|+|S_2|+|S_3|=d_{\max}\leqslant47$. This implies that there exists a vertex $y$ satisfies $y\not\sim x$ and $y$ has at least one neighbour in $S_i$, for $i=1, 2, 3$. This shows the claim.
\qed

Let $K$ be an induced subgraph of $G$ isomorphic to $K_{2,3}$.
As $K_{2,3}(19)\in\mathcal{F}_{-5}$, by Lemma \ref{forbidden}, we find that $\alpha(R(K))\leqslant18$. Since $\rho(R(K))<2$, we have $n_{R(K)}\leqslant2\alpha(R(K))\leqslant2\times18=36$. This means that
\begin{align*}
n&\leqslant n_K(d_{\max}+1)-2\varepsilon_{K}+n_{R(K)}\\
&\leqslant5(47+1)-12+36\\
&=264,
\end{align*}
a contradiction. This finishes the proof of this theorem.
\end{proof}

\section{A new bound for the independence number}
\label{sec:29}
We start with the following result.

\begin{theorem}
Let $S$ be a Seidel matrix with $\lambda_{\min}(S)=-5$ of order $n\geqslant277$. If the clique number $\omega([S])$ of $[S]$ satisfies $\omega([S])\in \{2, 3, 5, 6\}$, then {\rm rk}$(S+5\mathbf{I})\geqslant\frac{2n}{3}+1$.
\end{theorem}
Indeed, the case $\omega([S])=6$ was shown by Lemmens and Seidel \cite{lemmens1973};
the case $\omega([S]) = 2$ is trivial, as shown in Section \ref{sec:49};
and the cases $\omega([S])\in\{3,5\}$ were shown by Lin and Yu \cite{lin2018equiangular}.
Note that the cases $\omega([S])\leqslant3$ also follow from Theorem \ref{LStrianglefree}, since one can always isolate a vertex and the rest of the graph is triangle-free.

In this section we will show a new bound for the independence number for the case $\omega([S])=4$.

\begin{theorem}\label{LS38}
Let $S$ be a Seidel matrix of order $n\geqslant277$ with $\lambda_{\min}(S)=-5$ and $\omega([S])=4$. If the independence number $\alpha([S])$ of $[S]$ satisfies $\alpha([S])\geqslant39$, then {\rm rk}$(S+5\mathbf{I})\geqslant\frac{2n}{3}+1$.
\end{theorem}
\begin{proof}
Let $\alpha:=\alpha([S])$. By Theorem \ref{LS49}, we may assume $\alpha\leqslant48$. Let $G$ be a graph in the switching class of $S$ with $\alpha(G)=\alpha$. Let $C$ be an independent set of $G$ of order $\alpha$. We may assume that every vertex outside $C$ has at most $\lfloor\frac{\alpha}{2}\rfloor$ neighbours in $C$. Note that $G$ contains a triangle say with vertices $x$, $y$ and $z$, by Theorem \ref{LStrianglefree}. The set $V:=\{u\in V(G)\setminus\{x,y,z\}\mid u\not\thicksim x, u\not\thicksim y, u\not\thicksim z\}$ is an independent set, and $|V|\leqslant\alpha-1\leqslant47$, as $\omega([S])=4$ and $\alpha\leqslant48$. Without loss of generality, we may assume that the valency of $x$, $d_x$, is at least
\begin{equation*}
\lceil\frac{n-|V|-3}{3}+2\rceil\geqslant\frac{277-47-3}{3}+2=78,
\end{equation*}
as $n\geqslant277$. This implies the following claim.

\begin{claim}\label{3graphs}
The graph $G$ contains one of the following graphs as an induced subgraph.
\begin{figure}[ht]
\centering
\begin{tikzpicture}
    \draw (1,-0.5) node {$B_1$};
    \draw (6,-0.5) node {$K_{2,1,1}$};
    \draw (10.5,-0.5) node {$K_4$};
    \tikzstyle{every node}=[draw,circle,fill=white,minimum size=2pt,
                            inner sep=0pt]
                            {every label}=[\tiny]

    \draw (0,0) node (11) [label=below:$ $] {}
        -- ++(90:1cm) node (12) [label=below:$ $] {}
        -- ++(-30:1cm) node (13) [label=below:$ $] {}
        -- ++(30:1cm) node (14) [label=below:$ $] {}
        -- ++(-90:1cm) node (15) [label=below:$ $] {};
    \draw  (13) -- (15);
    \draw  (11) -- (13);

    \draw (4.5,0.5) node (0a) [label=left:$ $] {}
         -- ++(20:1.5cm) node (0b) [label=above:$ $] {}
         -- ++(-20:1.5cm) node (0c) [label=right:$ $] {}
         -- ++(200:1.5cm) node (0d) [label=below:$ $] {}
         --(0a);
    \draw (0b) -- (0d);

    \draw (10,0) node (0e) [label=left:$ $] {}
         -- ++(90:1cm) node (0f) [label=left:$ $] {}
         -- ++(0:1cm) node (0g) [label=right:$ $] {}
         -- ++(270:1cm) node (0h) [label=right:$ $] {}
         --(0e);
    \draw (0e) -- (0g);
    \draw (0f) -- (0h);
\end{tikzpicture}
\end{figure}
\end{claim}

\noindent{\bf {Proof of Claim \ref{3graphs}:}} Let $N_x$ be the set of all neighbours of $x$ in $G$. Let $N$ denote the subgraph of $G$ induced on $N_x$.
Since $N_x\geqslant78>48\geqslant\alpha$, $N$ contains an edge $uv$.
If the edge $uv$ is isolated, then the subgraph induced by $N_x \setminus \{u, v\}$ must contain another edge, and $N$ contains a $2K_2$.
If the edge $uv$ is not isolated, then $N$ contains a $K_{1,2}$ or a $K_3$.
Putting back the vertex $x$ implies the claim.
\qed

\begin{claim}\label{claim38}
Let $H$ be the induced subgraph as in Claim $\ref{3graphs}$.
\begin{enumerate}[(i)]
  \item If $H$ is isomorphic to $B_1$, then $\alpha\leqslant38$;
  \item If $H$ is isomorphic to $K_{2,1,1}$, then $\alpha\leqslant30$;
  \item If $H$ is isomorphic to $K_4$, then $\alpha\leqslant28$.
\end{enumerate}
\end{claim}
\noindent{\bf {Proof of Claim \ref{claim38}:}} We first show $(i)$. Assume that $G$ contains an induced subgraph $H$ isomorphic to $B_1$ and $\alpha\geqslant39$. Note that every vertex outside $C$ has at most $5$ neighbours in $C$, as $K_{1,t}(39-t) \in \mathcal{F}_{-5}$ unless $t\leqslant5$, by Lemma \ref{36} $(i)$. Then, the number of vertices in $C$ that are neither vertices of $H$ nor have a neighbour in $H$ is at least $\alpha-5n_H\geqslant39-25=14$. By Lemma \ref{forbidden2}, we have $B_1(14)\in \mathcal{F}_{-5}$ and this gives a contraction. This shows case $(i)$.

$(ii)$ and $(iii)$ follow in similar manner as $(i)$. In $(ii)$, we use that $K_{1,8}(14)$, $K_{1,7}(17)$, $K_{1,6}(23)$ and $K_{2,1,1}(11)$ are all in $\mathcal{F}_{-5}$, by Lemma \ref{36} $(i)$ and Lemma \ref{forbidden}. In $(iii)$, we use that $K_{1,8}(14)$, $K_{1,7}(17)$, $K_{1,6}(23)$ and $K_4(5)$ are all in $\mathcal{F}_{-5}$, by Lemma \ref{36} $(i)$ and Lemma \ref{forbidden}. This shows the claim.
\qed

This implies that, if $\alpha\leqslant 48$, then we have $\alpha\leqslant 38$ or {\rm rk}$(S+5\mathbf{I})\geqslant\frac{2n}{3}+1$. By Theorem \ref{LS49}, the proof of Theorem~\ref{LS38} is now finished.
\end{proof}

Once the bound $\alpha([S]) \leqslant 39$ is shown, we may use it again to further slash this bound, as the following theorem shows.

\begin{theorem}
Let $S$ be a Seidel matrix of order $n\geqslant277$ with $\lambda_{\min}(S)=-5$ and $\omega([S])=4$. If the independence number $\alpha([S])$ of $[S]$ satisfies $\alpha([S])\geqslant29$, then {\rm rk}$(S+5\mathbf{I})\geqslant\frac{2n}{3}+1$.
\end{theorem}
\begin{proof}
Let $\alpha:=\alpha([S])$. By Theorem \ref{LS38}, we may assume $\alpha\leqslant38$. Take a graph $G$ in the switching class of $S$ with $\alpha(G)=\alpha$. Let $C$ be an independent set of $G$ of order $\alpha$. We may assume that every vertex outside $C$ has at most $\lfloor\frac{\alpha}{2}\rfloor$ neighbours in $C$. By Theorem \ref{LStrianglefree}, it follows that $G$ contains a triangle, say with vertices $x$, $y$ and $z$. The set $V:=\{u\in V(G)\setminus\{x,y,z\}\mid u\not\thicksim x, u\not\thicksim y, u\not\thicksim z\}$ is an independent set, and $|V|\leqslant\alpha-1\leqslant37$, as $\omega([S])=4$ and $\alpha\leqslant38$. Without loss of generality, we may assume the valency $d_x$ of $x$ is at least
\begin{equation*}
\lceil\frac{n-|V|-3}{3}+2\rceil\geqslant\lceil\frac{277-37-3}{3}+2\rceil=81>2\alpha,
\end{equation*}
as $n\geqslant277$.
Therefore the subgraph $N_x$ induced by the neighbors of $x$ contains an edge which is not isolated, and
this implies that $N_x$ (and $G$) contains $K_{2,1,1}$ or $K_4$ as an induced subgraph.

\begin{claim}\label{28}
Assume that $G$ contains an induced subgraph $H$ isomorphic to $K_{2,1,1}$ or $K_4$. Then the following hold.
\begin{enumerate}[(i)]
  \item If $H$ is isomorphic to $K_{2,1,1}$, then $\alpha\leqslant28$;
  \item If $H$ is isomorphic to $K_4$, then $\alpha\leqslant24$.
\end{enumerate}
\end{claim}
\noindent{\bf Proof of Claim \ref{28}:} $(i)$ Assume that $G$ contains an induced subgraph $H$ isomorphic to $K_{2,1,1}$ and $\alpha\geqslant29$. Then every vertex outside $C$ has at most $5$ neighbours in $C$, as $K_{1,8}(14)$, $K_{1,7}(17)$ and $K_{1,6}(23)$ are all in $\mathcal{F}_{-5}$, by Lemma \ref{36} $(i)$. Note that the number of vertices in $C$, that are neither in $H$ nor have a neighbour in $H$, is at most $10$, as $K_{2,1,1}$
Then all vertices of $H$ have at least $\alpha-10\geqslant 29-10=19$ and at most $5n_H=20$ neighbours in $C$; in particular, none of the vertices in $H$ belongs to $C$.
It follows that there exists a vertex $u$ of $H$ with valency $3$ that has two neighbours, say $v$ and $w$, in $C$ that are not adjacent to any of the other three vertices of $H$. This means that the subgraph $H'$ induced on $V(H)$ $\cup$ $\{v,w\}$ is isomorphic to $B_2$ of Lemma \ref{forbidden2}. As $B_2(9)\in\mathcal{F}_{-5}$, by Lemma \ref{forbidden2}, this gives a contraction. This shows case $(i)$.

$(ii)$ Assume that $G$ contains an induced subgraph $H$ isomorphic to $K_4$ and $\alpha\geqslant25$. Note that every vertex outside $C$ has at most $6$ neighbours in $C$, as $K_{1,9}(12)$, $K_{1,8}(14)$ and $K_{1,7}(17)$ are in $\mathcal{F}_{-5}$, by Lemma \ref{36} $(i)$.
So, as $\alpha-6n_H\geqslant25-6\times 4=1$, there is at least 1 vertex in $C$ that is neither in $H$ nor has a neighbour in $C$. This implies that $G$ contains $K_4(1)$ as an induced subgraph, which gives a contradiction, as $\omega([S])=4$. This finishes the proof of case $(ii)$.
\qed

This implies that, if $\alpha\leqslant 38$, then we have $\alpha\leqslant 28$ or {\rm rk}$(S+5\mathbf{I})\geqslant\frac{2n}{3}+1$. Now the theorem immediately follows from Theorem \ref{LS38}.
\end{proof}

\section{Pillar}\label{sec:pillar}

Let $S$ be a Seidel matrix, and $\omega:=\omega([S])$ be the clique number of $[S]$.
Take a graph $G$ in the switching class of $S$ with $\omega(G)=\omega$.

Let $B:=\{x_1,\ldots,x_{\omega}\}$ be the vertex set of an $\omega$-clique inside $G$. We call $B$ a base of order $\omega$. For $U\subseteq B$, the pillar $\mathcal{P}_U$ with respect to $B$ is the set $\{y\notin B\mid y\thicksim u,$ if $u\in U$, $y\not\thicksim u,$ if $u\in B-U\}$ of vertices. Without loss of generality, we may assume $|U|\leqslant\lfloor\frac{\omega}{2}\rfloor$ and if $|U|=\frac{\omega}{2}$, then $x_1\in U$. Note that $\mathcal{P}_\emptyset = \emptyset$, as otherwise, $\omega([S])\geqslant\omega+1$. Let $p_U$ denote the cardinality of $\mathcal{P}_U$. Let $p_{\omega,t}$ denote the maximum cardinality of $\mathcal{P}_U$, where $|U|=t\leqslant\lfloor\frac{\omega}{2}\rfloor$. We call that $\mathcal{P}_U$ is a $(\omega,t)$ pillar when $|U| = t$.

Let $S$ be a Seidel matrix with $\lambda_{\min}(S)\geqslant-5$. In the rest of the article we will show that for the case $\omega([S]) = 4$, the Theorem \ref{maintheorem} is true.

Let $S$ be a Seidel matrix of order $n\geqslant277$ with $\lambda_{\min}(S)=-5$ and $\omega([S])=4$. Take a graph $G$ in $[S]$ with vertex set $V(G)$ and $\omega(G)=4$. Let $\{x_1,\ldots,x_4\}$ be a $4$-base. Let $S':=S(G)$. Note that $S'+5\mathbf{I}$ is positive semidefinite, so there exists a map $\tau:\{1,\ldots,n\}\mapsto\mathbb{R}^t:x\mapsto \hat{x}$ for some positive integer $t$ such that $\langle\hat{x},\hat{y}\rangle=(S'+5\mathbf{I})_{xy}$ for $x,y\in V(G)$. Let $W$ be the vector space spanned by $\{\hat{x_1},\ldots,\hat{x_4}\}$.

\subsection{A new bound for \texorpdfstring{$p_{4,1}$}{p(4,1)}}

Let $y\in \mathcal{P}_{\{x_1\}}$. Decompose $\hat{y}$ into $\hat{y}=h_1+c_y$ such that $h_1\in W$ and $c_y\in W^\bot$. Then, as $\langle\hat{y},\hat{x_1}\rangle=-1$ and $\langle\hat{y},\hat{x_i}\rangle=1$ for $i=2,3,4$, it follows that $h_1=\frac{1}{3}(\hat{x_2}+\hat{x_3}+\hat{x_4})$ and $\langle c_y,c_y\rangle=4$. Let $y'\in \mathcal{P}_{\{x_1\}}-\{y\}$. Then $\langle c_y,c_{y'}\rangle=0$.
Likewise, for $u\in \mathcal{P}_{\{x_2\}}$, $\hat{u}$ can be decomposed into $\hat{u}=h_2+c_u$ where $h_2 = \frac{1}{3}(\hat{x_1}+\hat{x_3}+\hat{x_4})\in W$ and $c_u\in W^\bot$.
Now $u\sim y$ implies that $\langle c_u,c_y\rangle=-\frac{4}{3}$ and $u\not\sim y$ implies that $\langle c_u,c_y\rangle=\frac{2}{3}$.

Lin and Yu \cite{lin2018equiangular} showed that if $p_{4,1}\geqslant25$, then there exists exactly one $(4,1)$ pillar having more than one vertex. The next theorem gives a similar result.

\begin{theorem}\label{(4,1)pillar}
Let $S$ be a Seidel matrix with $\lambda_{\min}(S)=-5$ and $\omega([S])=4$. If there exists an edge between two different $(4,1)$ pillars, then each of the other two $(4,1)$ pillars contains at most $19$ vertices.
\end{theorem}

\begin{proof}
Without loss of generality, we may assume that $u\in \mathcal{P}_{\{x_2\}}, v\in \mathcal{P}_{\{x_3\}}$ and $u$ and $v$ are adjacent. Let $\mathcal{P}_{\{x_1\}}=\{y_1,\ldots,y_p\}$, where $p$ is the number of vertices in $\mathcal{P}_{\{x_1\}}$. Write
\begin{align*}
  \hat{y_i}=h_1+c_{y_i}, & \text{~where } h_1=\frac{1}{3}(\hat{x_2}+\hat{x_3}+\hat{x_4}) \text{ for } i=1,\ldots,p;\\
  \hat{u}=h_2+c_u, & \text{~where } h_2=\frac{1}{3}(\hat{x_1}+\hat{x_3}+\hat{x_4}); \\
  \hat{v}=h_3+c_v, & \text{~where } h_3=\frac{1}{3}(\hat{x_1}+\hat{x_2}+\hat{x_4}).
\end{align*}

The Gram matrix $M$ of $c_{y_1},\ldots,c_{y_p},c_u,c_v$ is
\begin{equation*}
M= \left(
\begin{array}{c@{}c@{}}
 4\mathbf{I}_p  & B \\
 B^T & \begin{array}{rr}
 4 & -\frac{4}{3}\\
 -\frac{4}{3} &4
 \end{array}\\
 \end{array}\right),
\end{equation*}
where $B$ is a $(p\times2)$-matrix with each row is one of $(\frac{2}{3},\frac{2}{3})$, $(\frac{2}{3},-\frac{4}{3})$, $(-\frac{4}{3},\frac{2}{3})$, and $(-\frac{4}{3},-\frac{4}{3})$, and we denote the number of occurrences of these rows by $\alpha_{00}$, $\alpha_{01}$, $\alpha_{10}$, and $\alpha_{11}$, respectively.

Let $S_1$, $S_2$, $S_3$, and $S_4$ be the sets of rows of $M$ which have an occurrence of $(\frac{2}{3},\frac{2}{3})$, $(\frac{2}{3},-\frac{4}{3})$, $(-\frac{4}{3},\frac{2}{3})$, and $(-\frac{4}{3},-\frac{4}{3})$, respectively. Let $\pi=\{S_1,\ldots,S_4, \{u\},\{v\}\}$ be a partition of $\mathcal{P}_{\{x_1\}}\cup\{u\}\cup\{v\}$. Observe that $\pi$ is an equitable partition of $M$.
The quotient matrix of $M$ with respect to $\pi$ is given by
\begin{equation*}
Q= \frac{1}{3}
\begin{pmatrix}
 12 & 0 & 0 & 0 & 2 & 2 \\
 0 & 12 & 0 & 0 & 2 & -4 \\
 0 & 0 & 12 & 0 & -4 & 2 \\
 0 & 0 & 0 & 12 & -4 & -4 \\
 2\alpha_{00} & 2\alpha_{01} & -4\alpha_{10} & -4\alpha_{11} & 12 & -4 \\
 2\alpha_{00} & -4\alpha_{01} & 2\alpha_{10} & -4\alpha_{11} & -4 & 12
 \end{pmatrix}.
\end{equation*}
As $M$ is a positive semidefinite, all the eigenvalues of $Q$ must be non-negative, by Lemma \ref{quotient}.

Let $Q_i$ be the matrix obtained from $Q$ by removing the $i^{th}$ row and column of $Q$ for $i=5,6$. Then eigenvalue of $Q$ are all non-negative if and only if $\det(Q_5)\geqslant0$,  $\det(Q_6)\geqslant0$, and  $\det(Q)\geqslant0$, as $4\mathbf{I}$ is positive definite.

We find that $\det(Q_5)\geqslant0$ if and only if
\begin{equation}\label{Q5}
  \alpha_{00}+4\alpha_{01}+\alpha_{10}+4\alpha_{11}\leqslant36,
\end{equation}
and $\det(Q_6)\geqslant0$ if and only if
\begin{equation}\label{Q6}
  \alpha_{00}+\alpha_{01}+4\alpha_{10}+4\alpha_{11}\leqslant36.
\end{equation}
Formulae (\ref{Q5}) and (\ref{Q6}) imply
\begin{equation}\label{Q5Q6}
  \alpha_{00}+4\alpha_{11}\leqslant36-\frac{5}{2}(\alpha_{01}+\alpha_{10}).
\end{equation}
Furthermore, $\det(Q)\geqslant0$ if and only if
\begin{equation}\label{Q}
  3\alpha_{01}\alpha_{10}+(3(\alpha_{00}+4\alpha_{11})-44)(\alpha_{01}+\alpha_{10})-32(\alpha_{00}+4\alpha_{11}-12)\geqslant0.
\end{equation}

Define $\beta_1:=\alpha_{00}+4\alpha_{11}$ and $\beta_2:=\alpha_{01}+\alpha_{10}$. As $\alpha_{01}\alpha_{10}\leqslant(\frac{\beta_2}{2})^2$, equation (\ref{Q}) implies
\begin{equation}\label{simpleQ}
\frac{1}{4}(3\beta_2-32)(4\beta_1+\beta_2-48)\geqslant0.
\end{equation}
Equation (\ref{Q5Q6}) give
\begin{equation}\label{simpleQ5Q6}
  \beta_1\leqslant36-\frac{5}{2}\beta_2.
\end{equation}
We need to consider two cases $\beta_2\geqslant11$ and $\beta_2\leqslant10$. If $\beta_2\geqslant11$, then (\ref{simpleQ}) combined with (\ref{simpleQ5Q6}) gives
\begin{equation*}
  \frac{3}{4}(3\beta_2-32)(32-3\beta_2)\geqslant0.
\end{equation*}
This is a contradiction.

So $\beta_2\leqslant10$. Now (\ref{simpleQ}) implies $4\beta_1+\beta_2\leqslant48$. Hence, $p = \alpha_{00} + \alpha_{01} + \alpha_{10} + \alpha_{11}\leqslant\beta_1+\beta_2=\frac{(4\beta_1+\beta_2)+3\beta_2}{4}\leqslant\lfloor\frac{48+3\times10}{4}\rfloor=19$. This concludes the proof of this theorem.
\end{proof}
\begin{remark}
With some extra calculations, it can be shown that $p=19$ implies that $\alpha_{00}=9$, $\alpha_{01}=\alpha_{10}=5$ and $\alpha_{11}=0$.
This result can also be obtained by semidefinite integer programming if we follow the similar approach in Lin-Yu~ \cite{lin2018equiangular}.
\end{remark}

\subsection{A bound for \texorpdfstring{$p_{4,2}$}{p(4,2)}}
In this subsection, we are going to bound the order of a $(4,2)$-pillar.

Let $S$ be a Seidel matrix of order $n$ and $\lambda_{\min}(S)=-5$. Assume $\omega([S])=4$. Let $B:=\{x_1,\ldots,x_4\}$ be a $4$-base and take a graph $G$ in $[S]$ such that $\{x_1,\ldots,x_4\}$ is a clique.

Consider the $\{x_1,x_2\}$-pillar $\mathcal{P}_{\{x_1,x_2\}}$ with respect to $B$, say with order $p$. Let $G_{\mathcal{P}}$ be the subgraph of $G$ induced on $\mathcal{P}=\mathcal{P}_{\{x_1,x_2\}}$. Now switch $G$ with respect to $\{x_1,x_2\}$ to obtain $G_{sw}(\{x_1,x_2\})$. So the subgraph induced $\{x_1,\ldots,x_4\}\cup\mathcal{P}$ is $G_{\mathcal{P}}\dot\cup2K_2$. Note that $G_{\mathcal{P}}$ is triangle-free and $\alpha(G_{\mathcal{P}})\leqslant\alpha([S])-2$.

We will show the following result, and, as a corollary, we obtain a bound of the order of a $(4,2)$-pillar.

\begin{theorem}\label{pre(4,2)pillar}
Let $G$ be a triangle-free graph with order $n_G$ such that $\lambda_{\min}(S(G\dot\cup2K_2))\geqslant-5$. Assume further that $\alpha(G)\leqslant26$. For a vertex $x$ of $G$, let $a_x$ be the valency of $x$ in $G$. Then the following hold:
\begin{enumerate}[(i)]
  \item $n_G\leqslant68$;
  \item If $n_G\geqslant66$, then there exists an edge $xy$ in $G$ such that $a_x+a_y\leqslant20$.
\end{enumerate}
\end{theorem}

Before we give the proof of this theorem, we start with a few lemmas that can be verified by straightforward computations.
For $B_3$, see the picture below.

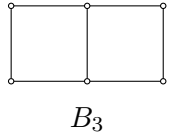
\begin{figure}[ht]
\centering
\begin{tikzpicture}
    \draw (1,-0.5) node {$B_3$};
    \tikzstyle{every node}=[draw,circle,fill=white,minimum size=2pt,
                            inner sep=0pt]
                            {every label}=[\tiny]

    \draw (0,0) node (11) [label=below:$$] {}
        -- ++(0:1cm) node (12) [label=below:$$] {}
        -- ++(0:1cm) node (13) [label=below:$$] {}
        -- ++(90:1cm) node (14) [label=above:$$] {}
        -- ++(180:1cm) node (15) [label=above:$$] {}
        -- ++(180:1cm) node (16) [label=above:$$] {};
    \draw  (12) -- (15);
    \draw  (11) -- (16);

\end{tikzpicture}
\end{figure}

\begin{lemma}\label{pillarforbidden1}
The following graphs are in $\mathcal{F}_{-5}$, where $P_n$ is a path with length $n-1$.
\begin{enumerate}[(i)]
  \item $K_{1,14}\dot\cup2K_2$;
  \item $H\dot\cup P_4\dot\cup K_2$, where $H\in\{K_{2,3},K_{1,6},B_3\}$;
  \item $H\dot\cup K_{1,3}\dot\cup K_2$, where $H\in\{K_{2,3},K_{1,6},B_3\}$.
\end{enumerate}
\end{lemma}

\begin{lemma}\label{pillarforbidden2}
For a graph $G$, let $G(t_1,t_2,t_3)$ be the disjoint union of $G$, $t_1$ isolated vertices, $t_2$ copies of $K_2$ and $t_3$ copies of $P_3$, where $t_1$, $t_2$ and $t_3$ are non-negative integers. Then, the following hold.
\begin{enumerate}[(i)]
  \item The smallest eigenvalue of $K_{2,3}(t_1,t_2+1,t_3)$ satisfies $\lambda_{\min}(S(K_{2,3}(t_1,t_2+1,t_3)))\geqslant-5$ if and only if $t_1+4t_2+10t_3\leqslant14$;
  \item The smallest eigenvalue of $K_{1,6}(t_1,t_2+2,t_3)$ satisfies $\lambda_{\min}(S(K_{1,6}(t_1,t_2+2,t_3)))\geqslant-5$ if and only if $t_1+4t_2+10t_3\leqslant14$;
  \item The smallest eigenvalue of $B_3(t_1,t_2+2,t_3)$ satisfies $\lambda_{\min}(S(B_3(t_1,t_2+2,t_3)))\geqslant-5$ if and only if $t_1+4t_2+10t_3\leqslant16$.
\end{enumerate}
\end{lemma}

\begin{lemma}\label{noK23}
Let $G$ be a triangle-free graph with order $n_G$ such that $\lambda_{\min}(G\dot\cup2K_2)\geqslant-5$. Let $a_{\max}$ be the maximum valency of $G$. If $n_G\geqslant66$, then
\begin{enumerate}[(i)]
  \item $a_{\max}\leqslant12$;
    \item $G$ does not contain $K_{2,3}$ as an induced subgraph.
\end{enumerate}
\end{lemma}
\begin{proof}
As $K_{1,14}\dot\cup2K_2\in \mathcal{F}_{-5}$, by Lemma \ref{pillarforbidden1} $(i)$, we have $a_{\max}\leqslant13$.
Suppose that $a_{\max}=13$, and let $x$ be a vertex with valency $13$ in $G$.
Consider a vertex $y\neq x$ which is not adjacent to $x$. Assume that $x$ and $y$ have $t$ common neighbours. Let $K$ be the subgraph of $G$ induced on $\{x,y\}\cup\{z\in V(G)\mid z\sim x\}$. We observe that $\det(S(K\dot\cup2K_2)+5\mathbf{I})\geqslant0$ if and only if $(t-3)^2\leqslant0$. It follows that every vertex, that is not adjacent to $x$, has $3$ common neighbours with $x$ in $G$. As $a_{\max}\leqslant13$, we have
\begin{align*}
n_G & \leqslant1+a_{\max}+\frac{a_{\max}(a_{\max}-1)}{3}\\
& \leqslant1+13+\frac{13(13-1)}{3}\\
& =66.
\end{align*}
By the assumption that $n_G\geqslant66$, we see that the equality must hold, and every vertex in $G$ has valency $13$.
Hence, in this case, $G$ is a strongly regular graph with parameters $(66,13,0,3)$. Such a strongly regular graph does not exist as the multiplicities of the eigenvalues are non-integral, by Lemma \ref{SRG}.
Hence $a_{\max} \leqslant 12$.

Suppose that $G$ contains $K_{2,3}$ as an induced subgraph, say $H$.
There are at most $n_{H}(1+a_{\max})-2\varepsilon_{H}$ vertices in $G$ that are either in $H$ or have a neighbour in $H$. Let $R(H)$ be the subgraph induced on the vertices of $G$ that are neither in $H$ nor have a neighbour in $H$. Now $R(H)$ has neither $P_4$ nor $K_{1,3}$ as an induced subgraph by Lemma \ref{pillarforbidden1}.
It follows by Lemma \ref{pillarforbidden2} $(i)$ that $R(H)$ has at most $10$ vertices, as there are 2 non-adjacent $K_2$ outside $R(H)$ in $G$. So
\begin{align*}
  n_G&\leqslant n_H(1+a_{\max})-2\varepsilon_H+n_{R(H)}\\
  &\leqslant5+5a_{\max}-12+10 \\
  &\leqslant 63 < 66,
\end{align*}
a contradiction.  Therefore $G$ may not contain a $K_{2,3}$ as an induced subgraph.
This finishes the proof.
\end{proof}

The next lemma is needed when $G$ does not contain a $K_{2,3}$ as an induced subgraph
and it can be proved by straightforward computations.

\begin{lemma}\label{pillaredgeneighbour}
Let $H$ be the disjoint union of $tK_2$ and $s_1+s_2$ isolated vertices, where $t$, $s_1$ and $s_2$ are non-negative integers. Let $M(s_1,s_2,t)$ be the graph obtained by adding an edge $xy$ to $H$ such that $x$ is adjacent to $t$ isolated vertices in $tK_2$ and $s_1$ isolated vertices, and $y$ is adjacent to all vertices in $H$ that are not adjacent to $x$. For example, see $M(3,2,3)$ as below. Then, the smallest eigenvalue of $M(s_1,s_2,t)\dot\cup2K_2$ satisfies $\lambda_{\min}(S(M(s_1,s_2,t)\dot\cup2K_2))\geqslant-5$ if and only if $3s_1+3s_2+4t\leqslant36$.
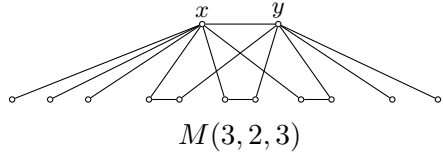
\begin{center}
\begin{tikzpicture}
\draw (0,-0.5) node {$M(3,2,3)$};
 \tikzstyle{every node}=[draw,circle,fill=white,minimum size=2pt,
                            inner sep=0pt]
                            {every label}=[\footnotesize]
    \draw (-3,0) node (1) {};
    \draw (-2.5,0) node (2) {};
    \draw (-2,0) node (3) {};
    \draw (-1.2,0) node (4) {};
    \draw (-0.8,0) node (5) {};
    \draw (-0.2,0) node (6) {};
    \draw (0.2,0) node (7) {};
    \draw (0.8,0) node (8) {};
    \draw (1.2,0) node (9) {};
    \draw (2,0) node (10) {};
    \draw (2.6,0) node (11) {};
    \draw (-0.5,1) node (13) [label=above:$x$] {};
    \draw (0.5,1) node (14) [label=above:$y$] {};
    \draw (4) -- (5);
    \draw (6) -- (7);
    \draw (8) -- (9);
    \draw (13) -- (14);
    \draw (13) -- (1);
    \draw (13) -- (2);
    \draw (13) -- (3);
    \draw (13) -- (4);
    \draw (13) -- (6);
    \draw (13) -- (8);
    \draw (14) -- (5);
    \draw (14) -- (7);
    \draw (14) -- (9);
    \draw (14) -- (10);
    \draw (14) -- (11);
\end{tikzpicture}
\end{center}
\end{lemma}

\begin{remark}
  $M(0,0,2)$ is the graph $B_3$.
  Therefore if a graph does not contain $B_3$, it does not contain $M(0,0,2)$ either.
\end{remark}

\begin{lemma}\label{nG68}
Let $G$ be a triangle-free graph with order $n_G\geqslant66$ such that $\lambda_{\min}(G\dot\cup2K_2)\geqslant-5$. For a vertex $x$ in $G$, let $a_x$ be the valency of $x$ in $G$. If the independence number $\alpha(G)$ of $G$ satisfies $\alpha(G)\leqslant26$, then
\begin{enumerate}[(i)]
  \item $a_x+a_y\leqslant20$ for all $xy\in E(G)$;
  \item $n_G\leqslant68$.
\end{enumerate}
\end{lemma}
\begin{proof}
$(i)$ Let $xy$ be an edge of $G$. Let $N(xy)$ be the subgraph of $G$ induced on the vertices outside $\{x,y\}$ and have a neighbour in $\{x,y\}$. As $G$ is triangle-free and does not contain $K_{2,3}$ as an induced subgraph, we observe that $N(xy)$ is isomorphic to $M(s_1,s_2,t)$ of Lemma \ref{pillaredgeneighbour}, for some non-negative integers $s_1$, $s_2$ and $t$. By Lemma \ref{pillaredgeneighbour}, $3s_1+3s_2+4t\leqslant36$. Note that
\begin{align*}
  a_x+a_y &=2+s_1+s_2+2t\\
   &=2+\frac{2s_1+2s_2+4t}{2} \\
   &\leqslant2+\frac{36-s_1-s_2}{2} \\
   &\leqslant20.
\end{align*}
This shows the case $(i)$.

$(ii)$ We denote by $\rho(G)$ the spectral radius of $G$. If $\rho(G)\leqslant2$, then $n_G\leqslant\frac{5}{2}\alpha(G)\leqslant\frac{5\times26}{2}=65$, which is a contradiction. This implies that $a_{\max}\geqslant3$. By Lemma \ref{noK23}, we have $3\leqslant a_{\max}\leqslant12$. Next, we will consider two cases $6\leqslant a_{\max}\leqslant12$ and $3\leqslant a_{\max}\leqslant5$.

In the following proof, we will denote by $N(H)$ (resp. $R(H)$) subgraph of $G$ induced on the vertices outside $V(H)$ that have a neighbour (resp. no neighbours) in $H$, for an induced subgraph, $H$, of $G$.

\noindent{\bf Case $1$. $6\leqslant a_{\max}\leqslant12$.} First we assume that $G$ does not contain $B_3$ of Lemma \ref{pillarforbidden1} as an induced subgraph. Note that the graph $M(s_1,s_2,t)$ in $(i)$ satisfies that $t\leqslant1$ since $M(0,0,2)$ is the graph $B_3$.
It follows that
\begin{align*}
  a_x+a_y & =2+s_1+s_2+2t \\
   & =2+\frac{3s_1+3s_2+6t}{3} \\
   & \leqslant\lfloor2+\frac{36+2t}{3}\rfloor \\
   & =14.
\end{align*}
Let $v$ be a vertex in $G$ with valency $a_v=a_{\max}$ and $v_1,\ldots,v_6$ be $6$ neighbours of $v$. Then $a_v+a_{v_i}\leqslant14$ for $i=1,\ldots,6$. Let $H_1$ be the subgraph of $G$ induced on $\{v,v_1,\ldots,v_6\}$. Then
\begin{align*}
n_{N(H_1)}&\leqslant a_v+\sum\limits_{i=1}^6a_{v_i}-2\times6\\
&=\sum\limits_{i=1}^6(a_v+a_{v_i})-5a_v-12\\
&\leqslant6\times14-5\times6-12\\
&=42.
\end{align*}
Note that $R(H_1)$ has neither $P_4$ nor $K_{1,3}$ as an induced subgraph by Lemma \ref{pillarforbidden1}. It follows by Lemma \ref{pillarforbidden2} $(ii)$ that $R(H_1)$ has at most $14$ vertices, as $\rho(R(H_1))<2$. It follows that $n_G=n_{H_1}+n_{N(H_1)}+n_{R(H_1)}\leqslant 7+42+14=63$, which is a contradiction.

So $G_{\mathcal{P}}$ contains $B_3$ as an induced subgraph. Let $H_2\cong B_3$ be an induced subgraph of $G$, and $V(H_2)=\{w_1,\ldots,w_6\}$ and $E(H_2)=\{w_1w_2,w_3w_4,w_5w_6,w_1w_3,w_2w_4,w_3w_5,w_4w_6\}$. Note that $a_{w_i}+a_{w_{i+1}}\leqslant20$ for $i=1,3,5$, by $(i)$. Then
\begin{align*}
n_{N(H_2)}&\leqslant \sum\limits_{i=1}^6a_{w_i}-2\times7\\
&\leqslant 3\times20-2\times7\\
&=46.
\end{align*}
Note that $R(H_2)$ has neither $P_4$ nor $K_{1,3}$ as an induced subgraph by Lemma \ref{pillarforbidden1}. It follows by Lemma \ref{pillarforbidden2} $(iii)$ that $R(H_2)$ has at most $16$ vertices, as $\rho(R(H_2))<2$. So $n_G=n_{H_2}+n_{N(H_2)}+n_{R(H_2)}\leqslant6+46+16=68$. This shows the case $6\leqslant a_{\max}\leqslant12$.

\noindent{\bf Case $2$. $3\leqslant a_{\max}\leqslant5$.} Let $K$ be a minimal subgraph of $G$ with $\rho(K)>2$. Then, $\rho(R(K))<2$, by Proposition \ref{prop:one-large}. It follows that $n_{R(K)}\leqslant 2\alpha(R(K))\leqslant2(\alpha(G)-\alpha(K))\leqslant 52-2\alpha(K)$, as $\alpha(G)\leqslant26$.
So, $n_G\leqslant n_{K}+n_{K}a_{\max}-2\varepsilon_{K}+2\alpha(R(K))\leqslant n_{K}(1+a_{\max})-2\varepsilon_{K}+52-2\alpha(K)$.

First we assume that $3\leqslant a_{\max}\leqslant4$. If $G$ contains $\tilde{D}^+_4$ as an induced subgraph, say $H_3$. Then $n_{R(H_3)}\leqslant52-2\alpha(H_3)\leqslant52-2\times4=44$. Note that $n_{N(H_3)}\leqslant6a_{\max}-2\times5\leqslant6\times4-10=14$. So, $n_G=n_{H_3}+n_{N(H_3)}+n_{R(H_3)}\leqslant6+44+14=64$, this is a contradiction. Assume that $G$ does not contain $\tilde{D}^+_4$ as an induced subgraph. Let $H_4$ be a minimal subgraph of $G$ with $\rho(H_4)>2$. We observe that every vertex in $H_4$ has valency at most $3$ in $G$. In this case, we have $n_{H_4}\leqslant10$, $\varepsilon_{H_4}\geqslant n_{H_4}-1$ and $\alpha(H_4)\geqslant \lceil\frac{n_{H_4}}{2}\rceil$, by Corollary \ref{>2}. This shows that
\begin{align*}
  n_G & =n_{H_4}+n_{N(H_4)}+n_{R(H_4)} \\
   & \leqslant n_{H_4}(1+a_{\max})-2\varepsilon_{H_4}+52-2\alpha(H_4) \\
   & \leqslant 4n_{H_4}-2(n_{H_4}-1)+52-n_{H_4} \\
   & \leqslant n_{H_4}+54\\
  & \leqslant 64,
\end{align*}
this is a contradiction.

This follows that $a_{\max}=5$. Let $H_5\cong K_{1,5}$ be an induced subgraph of $G$. We obtain that
\begin{align*}
  n_G & =n_{H_5}+n_{N(H_5)}+n_{R(H_5)} \\
   & \leqslant n_{H_5}(1+a_{\max})-2\varepsilon_{H_5}+52-2\alpha(H_5) \\
   & \leqslant 6+6\times5-2\times5+52-2\times5 \\
  & =68.
\end{align*}
This finishes the proof of $(ii)$.
\end{proof}

Theorem \ref{pre(4,2)pillar} follows immediately from Lemma \ref{noK23} and \ref{nG68}. As an easy consequence of Theorem \ref{pre(4,2)pillar}, we obtain a bound of the order of $(4,2)$ pillar.

\begin{corollary}\label{(4,2)pillar}
Let $S$ be a Seidel matrix with $\lambda_{\min}(S)=-5$. Assume that $\omega([S])=4$ and $\alpha([S])\leqslant28$. Let $G$ be a graph in the switching class of $S$ such that $\omega(G)=4$. Let $\mathcal{P}$ be a $(4,2)$ pillar. Let $G_{\mathcal{P}}$ be the subgraph of $G$ induced on $\mathcal{P}$. Let $a_{x}$ be the valency of $x\in V(G_{\mathcal{P}})$ in $G_{\mathcal{P}}$. Let $p$ denote the number of vertices in $\mathcal{P}$. If $p\geqslant66$, then the following hold.
\begin{enumerate}[(i)]
  \item $a_x+a_y\leqslant20$ for all $xy\in E(G_{\mathcal{P}})$.
  \item $p\leqslant68$.
\end{enumerate}
\end{corollary}

\section{Gallery}\label{sec:gallery}

Let $S$ be a Seidel matrix of order $n$ and $\omega([S])=4$. Let $H$ be a graph in $[S]$ such that there exist two adjacent vertices $x_1$ and $x_2$ satisfy that $x_1$ and $x_2$ have no common neighbours. Define the gallery with respect to $\{x_1,x_2\}$, $\mathcal {Ga}(x_1,x_2)$, as the subgraph of $H$ induced on $U(x_1,x_2):=\{y\in V(H)\mid y$ is not adjacent to $x_1$ nor $x_2\}$.

\begin{lemma}\label{gallery}
Let $u,v\in U(x_1,x_2)$ such that $u\sim v$, where $U(x_1,x_2)$ is defined as above. Then the following hold:
  \begin{enumerate}[(i)]
  \item The set $W_{\{u\}}:=\{y\in U(x_1,x_2)\mid y\sim u, y\not\sim v\} \setminus \{v\}$  is the $(4,1)$ pillar $\mathcal{P}_{\{v\}}$ with respect to the $4$-base $\{u,v,x_1,x_2\}$.
  \item The set $W_{\emptyset}:=\{y\in U(x_1,x_2)\mid y\not\sim u, y\not\sim v\} \setminus \{u,v\}$ is the $(4,2)$ pillar $\mathcal{P}_{\{u,v\}}$ with respect to the $4$-base $\{u,v,x_1,x_2\}$.
  \item $U(x_1, x_2)$ is the disjoint union of $\{u,v\}$,  $W_{\{u\}}$, $W_{\{v\}}$, and $W_\emptyset$.
  \end{enumerate}
\end{lemma}

\begin{proof}
  This follows straightforward from the definition.
  In particular, (iii) follows from the assumption that $\omega([S])=4$.
\end{proof}

Now we come to our main result in this section.

\begin{theorem}\label{105}
  Let $S$ be a Seidel matrix of order $n\geqslant277$ with $\lambda_{\min}(S)=-5$, $\alpha([S])\leqslant28$ and $\omega([S])=4$. We define $\mathcal{Ga}(x_1,x_2)$ and $U(x_1,x_2)$ as above. For $y\in\mathcal{Ga}(x_1,x_2)$, let $b_y$ be the valency of $y$ in $\mathcal{Ga}(x_1,x_2)$. Let $b_{\max}$ be the maximum valency of $\mathcal{Ga}(x_1,x_2)$. Let $\kappa$ be the order of $\mathcal{Ga}(x_1,x_2)$. If $b_{\max}\leqslant20$, then the following hold:
\begin{enumerate}[(i)]
  \item $\kappa\leqslant105$;
  \item If $\kappa=105$, then there exists an edge $uv$ inside $\mathcal{Ga}(x_1,x_2)$ such that $b_u=b_v=20$.
\end{enumerate}
\end{theorem}
\begin{proof}
Assume that $\kappa\geqslant105$. Let $uv$ be an edge in $\mathcal{Ga}(x_1,x_2)$. As $b_{\max}\leqslant20$, note that
\begin{equation*}
  p_{\{u,v\}}=\kappa-(b_u+b_v)\geqslant105-2\times20=65.
\end{equation*}
Moreover, if equality holds, then we have $p_{\{u\}}=p_{\{v\}}=20$, $p_{\{u,v\}}=65$ and $\kappa=105$. Now we assume that $p_{\{u,v\}}\geqslant66$. Since $p_{\{u,v\}}>28\geqslant\alpha([S])$, there exist two adjacent vertices in $\mathcal{P}_{\{u,v\}}$, say $x$ and $y$. By Lemma \ref{gallery} and Corollary \ref{(4,2)pillar}, we have $p_{\{u,v\}}\leqslant 68$. It follows that
\begin{equation*}
b_x+b_y=\kappa-p_{\{u,v\}}\geqslant105-68=37.
\end{equation*}
Let $a_x$ and $a_y$ be the valencies of $x$ and $y$ in the subgraph of $\mathcal{Ga}(x_1,x_2)$ induced on $\mathcal{P}_{\{u,v\}}$, respectively. By Corollary \ref{(4,2)pillar}, we see that $a_x+a_y\leqslant20$. This means that $x$ and $y$ have $(b_x+b_y)-(a_x+a_y)\geqslant37-20=17$ neighbours outside $\mathcal{P}_{\{u,v\}}$ in $\mathcal{Ga}(x_1,x_2)$. Without loss of generality, we may assume that $u$ and $x$ have at least $\lceil\frac{17}{4}\rceil=5$ common neighbours in $\mathcal{Ga}(x_1,x_2)$. Let $w_1,\ldots,w_5$ are $5$ common neighbours of $u$ and $x$.

Let $K$ be the subgraph of $\mathcal{Ga}(x_1,x_2)$ induced on $\{u,x,w_1,w_2,w_3\}$. Then there are at most $5+b_u+b_x+b_{w_1}+b_{w_2}+b_{w_3}-2\times6-2=b_u+b_x+b_{w_1}+b_{w_2}+b_{w_3}-9$
vertices in $\mathcal{Ga}(x_1,x_2)$ that are either in $K$ or have a neighbour in $K$. Let $R(K)$ be the subgraph induced on the vertices of $\mathcal{Ga}(x_1,x_2)$ that are neither in $K$ nor have a neighbour in $K$. Now $R(K)$ has neither $P_4$ nor $K_{1,3}$ as an induced subgraph by Lemma \ref{pillarforbidden1}. By Lemma \ref{pillarforbidden2} $(i)$, we see that $R(K)$ has at most $14$ vertices, as $\rho(R(K))<2$. It follows that
\begin{align*}
\kappa & \leqslant b_u+b_x+b_{w_1}+b_{w_2}+b_{w_3}-9+14\\
& \leqslant5b_{\max}+5\\
& \leqslant5\times20+5\\
&=105,
\end{align*}
and if $\kappa=105$ holds, then we have $b_u=b_x=b_{w_1}=b_{w_2}=b_{w_3}=20$. This shows the theorem.
\end{proof}


We will show the following theorem in the remaining of this section. This finishes the proof of Theorem \ref{maintheorem}.

\begin{theorem}
  Let $S$ be a Seidel matrix of order $n$ and $\lambda(S)=-5$. If $\omega([S])=4$ and $\alpha([S])\leqslant28$, then $n\leqslant276$.
\end{theorem}
\begin{proof}
  Let $\mathcal{Ga}(x,y)$ be the gallery with respect to $xy$ with order $q_{xy}$. For a $4$-base $\{x_1,\ldots,x_4\}$ denote for $U\subseteq\{x_1,\ldots,x_4\}$ the $U$-pillar with respect to $\{x_1,\ldots,x_4\}$ by $\mathcal{P}_U$ with order $p_U$.

  Let $\{x_1,\ldots,x_4\}$ be a $4$-base. If $p_{\{x_1\}}\geqslant20$, then $p_{\{x_2\}}+p_{\{x_3\}}+p_{\{x_4\}}\leqslant\alpha([S])$, by Theorem \ref{(4,1)pillar}. It follows that
\begin{align*}
  n&=4+p_{\{x_1\}}+p_{\{x_2\}}+p_{\{x_3\}}+p_{\{x_4\}}+p_{\{x_1,x_2\}}+p_{\{x_1,x_3\}}+p_{\{x_1,x_4\}}\\
  &\leqslant 4+2\alpha([S])+3p_{4,2}\\
  &\leqslant 4+2\times28+3\times68 \\
  &=264,
\end{align*}
as $\alpha([S])\leqslant28$ and $p_{4,2}\leqslant68$.

Now we may assume that $p_{\{x_i\}}\leqslant19$ for $i=1,\ldots,4$. As $q_{x_1x_2}=2+p_{\{x_3\}}+p_{\{x_4\}}+p_{\{x_1,x_2\}}$ and $q_{x_3x_4}=2+p_{\{x_1\}}+p_{\{x_2\}}+p_{\{x_1,x_2\}}$, we observe that $q_{x_1x_2}+q_{x_3x_4}=4+p_{\{x_1\}}+p_{\{x_2\}}+p_{\{x_3\}}+p_{\{x_4\}}+2p_{\{x_1,x_2\}}$. By Theorem \ref{105}, we have $q_{x_1x_2}\leqslant105$ and, if $q_{x_1x_2}=105$, then we may assume that $p_{\{x_3\}}=p_{\{x_4\}}=19$ and hence $p_{\{x_1,x_2\}}=65$. It follows that
\begin{align*}
  n&=4+p_{\{x_1\}}+p_{\{x_2\}}+p_{\{x_3\}}+p_{\{x_4\}}+p_{\{x_1,x_2\}}+p_{\{x_1,x_3\}}+p_{\{x_1,x_4\}}\\
  &=\frac{1}{2}(q_{x_1x_3}+q_{x_2x_4}+q_{x_1x_4}+q_{x_2x_3})+p_{\{x_1,x_2\}}\\
  &\leqslant\frac{4\times105}{2}+65\\
  &=275.
\end{align*}
So we may assume that $q_{x_ix_j}\leqslant104$ for all $1\leqslant i<j\leqslant4$. We find
\begin{align*}
  n&=4+p_{\{x_1\}}+p_{\{x_2\}}+p_{\{x_3\}}+p_{\{x_4\}}+p_{\{x_1,x_2\}}+p_{\{x_1,x_3\}}+p_{\{x_1,x_4\}}\\
  &=\frac{1}{2}(q_{x_1x_3}+q_{x_2x_4}+q_{x_1x_4}+q_{x_2x_3})+p_{\{x_1,x_2\}}\\
  &\leqslant\frac{4\times104}{2}+68\\
  &=276.
\end{align*}

  This finishes the proof.
\end{proof}

\section*{Acknowledgments}
\indent

M.-Y. Cao is partially supported by the National Natural Science Foundation of China (No.~$11571044$ and No.~$61373021$) and the Fundamental Research Funds
for the Central Universities of China.

J. H. Koolen is partially supported by the National Natural Science Foundation of China (No.~$11471009$ and No.~$11671376$) and Anhui Initiative of Quantum
Information Technologies (No.~AHY $150000$).

Y.-C. R. Lin is partially supported by the Ministry of Science and Technology of Taiwan (No.~108-2115-M-003-002-).

W.-H. Yu is partially supported by the Ministry of Science and Technology of Taiwan (No.~107-2115-M-008-010-MY2)
\bibliographystyle{plain}
\bibliography{equiangular}

\end{document}